\documentclass[a4paper]{amsart}
\usepackage[utf8]{inputenc}
\usepackage[foot]{amsaddr}
\usepackage{amsmath,amsthm,amssymb}
\usepackage{todonotes}
\usepackage{mdframed}
\usepackage{mathtools}
\usepackage{empheq}
\usepackage{enumitem}
\usepackage{microtype}
\usepackage{mathtools}
\usepackage{bbm}
\usepackage[normalem]{ulem}
\usepackage{nicefrac}

\usepackage[style=numeric,isbn=false]{biblatex}
\AtEveryBibitem{
    \clearfield{urlyear}
    \clearfield{urlmonth}
}

\usepackage[bookmarks=false,hidelinks]{hyperref}
\bibliography{}

\newmdenv[linecolor=red,frametitle=TODO,linewidth=2]{todobox}

\newtheorem{theorem}{Theorem}
\newtheorem{lemma}[theorem]{Lemma}

\newtheorem{proposition}[theorem]{Proposition}
\newtheorem{corollary}[theorem]{Corollary}
\newtheorem{assumption}[theorem]{Assumption}

\theoremstyle{definition}
\newtheorem{definition}[theorem]{Definition}

\newtheorem{remark}[theorem]{Remark}
\newtheorem{example}[theorem]{Example}

\numberwithin{equation}{section}
\numberwithin{theorem}{section}

\newcommand{\R}{{\mathbb{R}}}
\newcommand{\N}{{\mathbb{N}}}
\newcommand{\Z}{{\mathbb{Z}}}
\newcommand{\loc}{{\mathrm{loc}}}
\DeclareMathOperator{\supp}{supp}
\DeclareMathOperator{\dist}{dist}
\newcommand{\from}{\colon}
\newcommand{\br}{\overline{\gamma}_r}
\DeclareMathOperator*{\esssup}{ess\, sup}

\DeclareMathOperator{\Div}{div}
\DeclareMathOperator{\sign}{sign}
\renewcommand{\phi}{\varphi}
\renewcommand{\epsilon}{\varepsilon}
\newcommand{\eps}{\epsilon}
\renewcommand{\leq}{\leqslant}
\renewcommand{\geq}{\geqslant}
\newcommand{\abs}[1]{\mathopen|#1\mathclose|}

\newcommand{\subscriptcdot}{\mbox{\normalsize$\cdot$}}

\begin{document}

\title{Transport equations for Osgood velocity fields}
\author[U.~S.~Fjordholm]{Ulrik Skre Fjordholm}
\address[U.~S.~Fjordholm]{Department of Mathematics\\
University of Oslo\\
Postboks 1053 Blindern, 0316 Oslo}
\email[]{ulriksf\@@{}math.uio.no}

\author[O.~M{\ae}hlen]{Ola M{\ae}hlen}
\address[O.~M{\ae}hlen]{Mathematical Institute of Orsay\\
Paris-Saclay University\\
91400 Orsay, France}
\email[]{ola.maehlen\@@{}universite-paris-saclay.fr}

\maketitle

\begin{abstract}
We consider the transport equation with a velocity field satisfying the Osgood condition. The weak formulation is not meaningful in the usual Lebesgue sense, meaning that the usual DiPerna--Lions treatment of the problem is not applicable {(in particular, the divergence of the velocity might be unbounded)}. Instead, we use Riemann--Stieltjes integration to interpret the weak formulation, leading to a well-posedness theory in regimes not covered by existing works. The most general results are for the one-dimensional problem, with generalisations to multiple dimensions in the particular case of log-Lipschitz velocities.
\end{abstract}
\section{Introduction}
Consider the transport equation
\begin{equation}\label{eq:transport}
\begin{cases}
\partial_t u + b\cdot\nabla u = 0 & \text{for } x\in\R^d, t>0 \\
u(\cdot,0) = u_0
\end{cases}
\end{equation}
for some $b\from\R^d\times\R_+\to\R^d$ and $u_0\from\R^d\to\R$. The link between \eqref{eq:transport} and the characteristic equation
\begin{equation}\label{eq:ode}
\dot{X}_t = b(X_t,t)
\end{equation}
is immediate when $b$ and $u_0$ are smooth and bounded, as the solution is given via the flow of \eqref{eq:ode}, $u(x,t) = u_0(X_0(x,t))$, where $t\mapsto X_t(x,\tau)$ is the unique solution of \eqref{eq:ode} satisfying $X_\tau = x$. This connection between \eqref{eq:transport} and \eqref{eq:ode} is still evident when regularity assumptions on $b$ are relaxed. Most prominently, DiPerna--Lions \cite{DiPLio89} proved well-posedness of \eqref{eq:transport} in the sense of distributions,
\begin{equation}\label{eq:transportWeak}
\int_0^\infty\int_{\R^d}u\partial_t\phi + u\Div(\phi b)\,dx\,dt + \int_{\R^d}u_0(x)\phi(x,0)\,dx = 0
\end{equation}
for all $\phi\in C_c^\infty(\R^d\times\R_+)$,
assuming $b\in W^{1,1}_x$ and with $\Div b\in L^\infty$. This was generalized to $b\in \mathrm{BV}$ with $\Div b\in L^\infty$ by Ambrosio~\cite{ambrosio_transport_2004}. The connection with \eqref{eq:ode} is here via so-called regular Lagrangian flows, which in a sense are ``almost everywhere'' solutions of \eqref{eq:ode} (see e.g.~\cite{CriDel08}).

Conversely, the well-posedness theory for \eqref{eq:ode} is quite separate from that of \eqref{eq:transport}, and there is a multitude of conditions that ensure existence and uniqueness for \eqref{eq:ode}, while allowing for irregularities or singularities in $b$ of different types (see e.g.~\cite{AgarwalLakshmi}).
These conditions are often formulated as pointwise regularity conditions, rather than Lebesgue integration-based conditions such as those in the DiPerna--Lions theory; a classical example is the one-sided \textit{Osgood condition}
\begin{equation}\label{eq: osgoodConditionForIntroduction}
    \langle b(x+y,t)-b(x,t), y/|y| \rangle\leq \omega(|y|),\qquad \text{where}\;\int_{0}^\eps\frac{dh}{\omega(h)}=\infty \;\; \forall\ \eps>0.
\end{equation}
It is a fact that if $b\in C_b(\R^d\times\R_+)$ satisfies this condition, then the ODE \eqref{eq:ode} is well-posed~\cite{AgarwalLakshmi}. However, \eqref{eq: osgoodConditionForIntroduction} offers no direct benefit for the PDE \eqref{eq:transportWeak}. In particular, \eqref{eq: osgoodConditionForIntroduction} does not imply that $\Div b$ exists, even as a measure, and the weak formulation \eqref{eq:transportWeak} may therefore be meaningless in the Lebesgue sense.

The purpose of this paper is to show that the Osgood condition \textit{does}, in fact, ensure well-posedness of the transport equation when we interpret the integrals in \eqref{eq:transportWeak} as Riemann--Stieltjes integrals.
Although we require some additional technical assumptions on $b$ and $u_0$, we stress that our assumptions do not imply $\Div b\in L^\infty$ nor $b\in \mathit{BV}$; thus, we are outside the scope of the DiPerna--Lions--Ambrosio theory.

Our most general results will be for the one-dimensional forwards equation
\begin{equation}\label{eq:transport1d}
\begin{cases}
\partial_t u + b\partial_x u = 0, \\
u(\cdot,0) = u_0
\end{cases}
\end{equation}
where the one-dimensionality will imply significantly improved regularity of the flow $X$ (in particular, monotonicity). For the multi-dimensional equation \eqref{eq:transport} we will generalise some of these results for \textit{log-Lipschitz} velocities --- a standard example of non-Lipschitz regularity that still satisfies the Osgood condition. The main body of this paper will therefore concentrate on the one-dimensional problem \eqref{eq:transport1d}, and we return to the multi-dimensional problem in Section \ref{sec:multi-d}. We also treat the one-dimensional backwards problem in Section~\ref{sec:backwards1d}, the vanishing viscosity problem in Section~\ref{sec:vanishing-viscosity}, and the inhomogeneous problem in Section~\ref{sec:inhomogeneous}.

\subsection{An overview of new ideas}\label{sec: sketchOfTheUniquenessArgument}
The weak interpretation \eqref{eq:transportWeak} of \eqref{eq:transport} makes sense granted $\Div b\in L^\infty(\R^d\times \R_+)$. In one spatial dimension, this constraint is particularly restrictive as $b(\cdot,t)$ must then be Lipschitz continuous (see e.g.~Bouchut and James~\cite{BouJam97} for a generalization to \emph{one-sided Lipschitz} $b$). We shall instead consider a weak formulation based on Riemann--Stieltjes integrals. Recall that the Riemann--Stieltjes integral of a function $f$ with respect to $g$ satisfies
\begin{equation}\label{eq: theObviousIdentityOfRiemannStieltjesIntegralInTheSmoothCase}
    \int_{\R} f(x)\, d_xg(x) = \int_{\R}f(x)\frac{dg}{dx}(x)\, dx
\end{equation}
whenever $f\in C^0(\R)$, $g\in C^1_c(\R)$. (The notation $d_xg$ is non-standard, but it will be useful to emphasize which variable we are integrating over.) In fact, the first integral in \eqref{eq: theObviousIdentityOfRiemannStieltjesIntegralInTheSmoothCase} exists for more general cases when neither $g'$ nor $f'$ exist as (say) locally integrable functions, as was first observed by L.~C.~Young \cite{Young1936}: The basic idea is that the Riemann--Stieltjes sums might still converge when one trades some of the regularity in $g$ for increased regularity in $f$. For instance, $f\in C^{0,\alpha}(\R)$ and $g\in C^{0,\beta}(\R)$ for $\alpha,\beta\in(0,1]$ satisfying $\alpha+\beta>1$ is sufficient.

Concentrating on the one-dimensional case first, our weak formulation of \eqref{eq:transport1d} is then
\begin{equation}\label{eq: introductionWeakRiemannStieltjesFormulation}
        \int_0^{\infty}\int_{\R} u \partial_t\phi\, dx\, dt + \int_0^{\infty}\int_{\R} u \,d_x (b \varphi) \, dt + \int_{\R} u_0\phi(x,0)\, dx = 0,
\end{equation}
for $\varphi\in C_c^\infty(\R\times [0,\infty))$, where $\int_{\R} u \,d_x (b \varphi) $ is interpreted as a Riemann–Stieltjes integral in $x$. The conditions that we will impose on $b$ and $u_0$ (cf.~Assumptions~\ref{ass:b-conditions} and \ref{ass:uzero-conditions}) --- in particular, the Osgood condition on $b$ --- are made precisely to make this integral well-defined. As it turns out, these conditions are also sufficient in order to make the Cauchy problem well-posed.

To contrast our approach to the established theory, DiPerna--Lions~\cite{DiPLio89} impose conditions on $b$ ensuring that the non-conservative product $b\partial_x u$ is meaningful in the sense of distributions, \emph{regardless of the regularity of $u$}.
In contrast, our conditions on $b$ (and $u_0$) do not yield an inherent meaning of the product $b\partial_x u$, but by taking into account the \emph{expected regularity} of the solution $u$, we can give meaning to the product --- effectively trading some regularity in $b$ for more regularity in $u$, in the spirit of Young~\cite{Young1936}.
A similar mechanism was mentioned by Lions and Seeger~\cite[Theorem 3.7]{lions_transport_2024}, but where the required H\"older regularity on $b$ and $u$ was hypothesised, rather than proved. We also mention Panov~\cite{panov_generalized_2008}, who interprets the one-dimensional transport equation in a duality sense (dually to the continuity equation). This interpretation is similar to the work of Bouchut and James~\cite{BouJam97}, but under somewhat weaker conditions.

Interpreting integrals of low regularity is the core aspect of \emph{rough path theory}, where additional information (``higher-order Taylor remainder terms'') is added to the Riemann--Stieltjes integrals in order to ensure their convergence and continuity. However, we stress that all integrands in our work will have sufficient regularity for a canonical interpretation of the Riemann--Stieltjes integrals, and rough path theory is therefore not required. We refer to~\cite{Hai2011} for a nice introduction and application to spatially rough SPDEs, as well as the comprehensive monographs~\cite{FriHai2014,FriVic2010}.

The existence of solutions of \eqref{eq: introductionWeakRiemannStieltjesFormulation} (Theorem \ref{thm:existence-of-solution}) follows from a standard mollification argument. Our uniqueness result, Theorem \ref{thm:uniqueness-and-stability}, is harder to prove and requires new techniques and estimates. To illustrate, we give here a brief description of our uniqueness proof and the main difficulties it overcomes.

Since the equation is linear, it suffices to prove uniqueness for the case $u_0\equiv0$. In the smooth case, one could multiply \eqref{eq:transport1d} with $\sign(u)$ (or a smoothened version of it) to conclude that also $|u|$ solved the equation. From there, one could integrate in $x$ to find that
\begin{equation}\label{eq: theL1ControlInSmoothCase}
    \partial_t \|u(t)\|_{L^1(\R)} = \int_{\R} |u(x,t)|\partial_xb(x,t)\, dx \leq \sup_{x} \partial_x b(x,t)\|u(t)\|_{L^1(\R)},
\end{equation}
and the sought conclusion $u=0$ would then follow from Gr\"onwall's inequality and the smoothness assumption, in particular $\partial_x b\in L^\infty$. This technique also works for one-sided Lipschitz velocities, where $\sup_x\partial_x b<\infty$; such velocity fields have been studied extensively both in one and multiple dimensions, and we refer to Bouchut--James~\cite{BouJam97}, Petrova--Popov~\cite{PetPop99}, Bianchini--Gloyer~\cite{BiaGlo10} and Lions--Seeger~\cite{lions_transport_2024}.

Although we are outside of the Lipschitz setting, Theorem \ref{thm:renormalizability} shows that the first part of \eqref{eq: theL1ControlInSmoothCase} still holds true here, where it reads
\begin{equation}\label{eq: introductionThePreGeneralizedGrönwallInequality}
    \partial_t \|u(t)\|_{L^1(\R)} = \int_{\R} |u(x,t)|\, d_xb(x,t).
\end{equation}
Of course, as $b$ now lacks Lipschitz continuity, we cannot proceed through Grönwall's inequality. A novel \emph{nonlinear} bound, assuming for a moment that $C\coloneqq \sup_{t>0}|u(t)|_{TV}<\infty$, is given in Proposition~\ref{prop: controlOnIntegralWithBVandOnesidedRegularity}:
\begin{align}\label{eq: alternativeBoundWhenBIsNotOneSidedLipschitz}
    \int_{\R}|u(x,t)|\, d_xb(x,t)\leq C\lambda(t)\omega_b\big(\|u(t)\|_{L^1}/C\big),
\end{align}
where $(t,h)\mapsto \lambda(t)\omega_b(h)$ is the assumed Osgood modulus of $b$.
Inserting this in \eqref{eq: introductionThePreGeneralizedGrönwallInequality} and applying the Osgood--Bihari--LaSalle inequality (in place of Gr\"onwall's inequality), we again get the desired conclusion that $\|u(t)\|_{L^1(\R)}=0$ for all $t\geq 0$.
But the $BV$-restriction on $u$ is undesirable, and not actually necessary, as it suffices to assume locally finite $p$-variation of $u$ for some arbitrary $p\in[1,\infty)$: One may then decompose $|u|$ in a Littlewood--Paley-like manner (cf.~Lemma~\ref{lem: monotoneDecomposition}), and then apply \eqref{eq: alternativeBoundWhenBIsNotOneSidedLipschitz} to each component, which leads to a more general $p$-variational version of \eqref{eq: alternativeBoundWhenBIsNotOneSidedLipschitz}, given by Theorem~\ref{thm: controlOnIntegralWithPVariationandOnesidedRegularity}.
From there, uniqueness again follows by the Osgood inequality. The technical machinery needed for this argument is developed in Section~\ref{sec:nonlinear-l1-estimate}.

This sums up the uniqueness argument in one dimension. We devote the remainder of this discussion to the multidimensional case $x\in\R^d$. Our weak formulation of \eqref{eq:transport} is then
\begin{equation}\label{eq: introductionWeakRiemannStieltjesFormulationSeveralVariables}
        \int_0^{\infty}\int_{\R^d} u \partial_t\phi\, dx\, dt + \int_0^\infty\int_{\R^{d}} u\,d_x(b\varphi) \,dt + \int_{\R^d} u_0\phi(x,0)\, dx = 0,
\end{equation}
for $\varphi\in C_c^\infty(\R^d\times [0,\infty))$, with the short-hand notation
\begin{equation}\label{eq: firstPlaceWhereXhatAppears}
\int_{\R^{d}} u\,d_x(b\varphi)  = \sum_{i=1}^d \int_{\R^{d-1}}\bigg(\int_{\R} u\,d_{x_i}(b_i\varphi)\bigg)\,d\hat{x}_i,
\end{equation}
and where $b=(b_1,b_2,\dots, b_d)$, $x=(x_1,x_2,\dots,x_d)$, and $\hat{x}_{i}\in \R^{d-1}$ is the $(d-1)$-tuple of entries in $x$ whose index differs from $i$, i.e.~$\hat{x}_i=(x_1,\dots, x_{i-1},x_{i+1},\dots,x_d)$. Several arguments from the one-dimensional setting carry over to the multi-dimensional one, since the $d$ Riemann--Stieltjes integrals from the transport term in \eqref{eq: introductionWeakRiemannStieltjesFormulationSeveralVariables} can be dealt with individually.

The true difficulty now is the lack of regularity one can expect from the solution $u$, thus jeopardising the convergence of the Riemann--Stieltjes integrals. Canonically, one expects $u$ to be given by $u(x,t)=u_0(X_t^{-1}(x))$, where $X_t^{-1}$ is the inverse of the flow $X_t$ of \eqref{eq:ode} (or put differently, the backwards flow of \eqref{eq:ode}), and where the forwards flow $X_t$ exists and is unique under the Osgood assumption on $b$. In one space dimension, the inverse map $X^{-1}_t$ is at least monotone, so that $u_0\mapsto u_0\circ X_t^{-1}$ maps, for example, $BV_{\loc}(\R)$ to itself. But in higher dimensions, this regularity is lost. However, the situation is manageable in the log-Lipschitz scenario where, essentially,
\begin{equation}\label{eq:log-lipschitz-condition_intro}
|b(x+y,t)-b(x,t)|\leq C|y|\log(1/|y|) \qquad\text{for }|y|\ll 1
\end{equation}
(see Assumption~\ref{ass:b-conditions-multi-d} for the precise condition), in which case the corresponding backwards flow turns out to be H\"older regular for bounded times. Regularity bounds like~\eqref{eq:log-lipschitz-condition_intro} appear, among other places, in the velocity field in the transport equation satisfied by the vorticity for two-dimensional incompressible Euler (see e.g.~Chapter~7 of Bahouri, Chemin, Danchin~\cite{bahouri_fourier_2011}, in particular Section~7.1.1). Log-Lipschitz regularity bounds (and, in particular, uniqueness of the corresponding flow) have also been derived for Sobolev regular velocity fields and, in particular, for solutions of the incompressible Navier--Stokes equations; see Zuazua~\cite{Zua2002}, Chemin--Lerner~\cite{CheLer1995}, Dashti--Robinson~\cite{DasRob2009} and the references therein. Note, however, that incompressible flows satisfy $\Div b=0$, and so there is no problem in interpreting $b\cdot \nabla u$ in the distributional sense. We mention also the recent work \cite{galeati_well-posedness_2025} which treats transport equations with an Osgood velocity field $b$, but with an additional transport term driven by a rough path. However, since they assume that the velocity is ``almost incompressible'', $\Div b \in L^\infty_x$, they again do not face the same issues in interpreting the product $b\cdot \nabla u$.

\subsection{Assumptions and solution concept}
We here present the context for the one-dimensional case; the multi-dimensional case is treated in Section \ref{sec:multi-d}.
\begin{definition}\label{def:one-sided-osgood}
An \emph{Osgood modulus} is a continuous function $\omega\colon [0,\infty)\to[0,\infty)$ such that $\omega(h)>0$ for $h>0$ and
 \begin{align}\label{eq: OsgoodCriterion}
     \int_0^{\epsilon}\frac{dh}{\omega(h)}=\infty
 \end{align}
for any $\eps>0$.
\end{definition}

\begin{example}\label{example:osgood-moduli}
The typical examples of Osgood moduli $\omega$ are
\begin{align*}
    \omega_0(h) \coloneqq h, \qquad
    \omega_1(h) \coloneqq h\abs{\log(h)}, \qquad
    \omega_2(h) \coloneqq h\abs{\log(h)}\abs{\log(\abs{\log(h)})},
\end{align*}
et cetera (all defined for $h\ll 1$, and extended linearly for $h\not\ll 1$).
The function $\omega(h)\coloneqq h^\alpha$ is not Osgood for any $\alpha<1$, nor is $\omega(h)\coloneqq h\abs{\log(h)}^\alpha$ for $\alpha>1$.
\end{example}

For the forwards, one-dimensional problem \eqref{eq:transport1d} we will make the following assumptions on $b$.
\begin{assumption}[The velocity field]\label{ass:b-conditions}
We assume that $b\from\R\times\R_+\to\R$ is measurable and that it satisfies the following properties.
\begin{enumerate}[label=(\Alph*)]
\item\label{ass:bounded} \emph{Boundedness:} $b$ is bounded.

\item\label{ass:cont} \emph{Spatial continuity:} $b(\cdot,t)\in C(\R)$ for a.e.~$t\in\R_+$.

\item\label{ass:osoc}  \emph{One-sided Osgood condition:} $b$ satisfies the one-sided estimate
 \begin{equation}\label{eq: theOsgoodCondition}
     b(x,t)-b(y,t)\leq \lambda(t)\omega_b(x-y),\qquad x>y,\quad \text{a.e.}\, t\in(0,\infty),
 \end{equation}
 for some weight $\lambda\in L^1_{\loc}([0,\infty))$ and some concave Osgood modulus $\omega_b$.

 \item\label{ass:oshc} \emph{One-sided Hölder continuity of orders less than one:} The Osgood modulus $\omega_b$ of $b$ satisfies the condition
 \begin{equation}\label{eq: oneSidedHölderContinuousOfAllOrdersLessThanOne}
     \sup_{h\in(0,1)}\frac{\omega_b(h)}{h^{\alpha}}<\infty,\quad \forall \alpha\in(0,1).
 \end{equation}
\end{enumerate}
\end{assumption}
\noindent As for the initial data, we set one regularity condition.
\begin{assumption}[The initial data]\label{ass:uzero-conditions}
We assume the following on $u_0$.
\begin{enumerate}[label=(E)]
\item\label{ass:regularityOfInitialData} \emph{Locally finite variation:} For every $R>0$ there is a $p\in[1,\infty)$ such that $|u_0|_{V^p([-R,R])}<\infty$, where $|\cdot|_{V^p}$ denotes the $p$-variation (see Section \ref{sec:p-variation}).
\end{enumerate}
\end{assumption}
 \noindent A detailed discussion regarding these assumptions can be found in Remark \ref{rem: remarkOnTheAssumptions} at the end of this section. Our solution concept is then as follows.
\begin{definition}[Weak solution]\label{def: definitionOfWeakSolution}
Let $b$ and $u_0$ satisfy Assumptions~\ref{ass:b-conditions} and \ref{ass:uzero-conditions}, respectively. We then say that $u\in L^\infty_{\loc}(\R\times [0,\infty))$ is a weak solution of \eqref{eq:transport1d} if the following two conditions hold.
\begin{enumerate}[label=(\roman*)]
\item\label{def: definitionOfWeakSolution-finite-variation}
It has locally finite variation: For every $R,T>0$ there is some $p\in[1,\infty)$ such that
\begin{equation}\label{eq: regularityAssumptionOnSolution}
  \esssup_{0\leq t<T} |u(t)|_{V^p([-R,R])} < \infty
\end{equation}
(where the $p$-variation is defined in Section \ref{sec:p-variation}).

\item\label{def: definitionOfWeakSolution-riemann-stieltjes}
The equation is satisfied weakly in the Riemann--Stieltjes sense: For every $\phi\in C_c^\infty(\R\times [0,\infty))$ we have
\begin{equation}\label{eq: theWeakRiemannStieltjesFormOfTheEquation}
    \int_0^{\infty}\int_{\R} u \partial_t\phi\, dx\, dt + \int_0^\infty\int_{\R} u\,d_x(b\varphi) \,dt + \int_{\R} u_0(x)\phi(x,0)\, dx = 0
\end{equation}
where $\int_{\R} u\,d_x(b\varphi)$ is interpreted as a Riemann--Stieltjes integral in the $x$ variable (see~Section~\ref{sec:riemann-stieltjes}).
\end{enumerate}
\end{definition}

The fact that the Riemann–Stieltjes integral in \eqref{eq: theWeakRiemannStieltjesFormOfTheEquation} is well-defined and finite, is explained in Section \ref{sec:riemann-stieltjes}; specifically, it follows from Corollary \ref{cor: finitePVariationOfTheVelocityB}, the regularity constraint \eqref{eq: regularityAssumptionOnSolution}, and Theorem \ref{thm:young-integral-properties} \ref{thm:young-integral-properties-well-defined}.

\begin{theorem}[Main Theorem, one dimension]\label{thm:forwards-problem-well-posed-1d}
    Let $b$ and $u_0$ satisfy Assumptions~\ref{ass:b-conditions} and \ref{ass:uzero-conditions}, respectively. Then, there exists a unique weak solution $u$ of \eqref{eq:transport1d} in the sense of Definition \ref{def: definitionOfWeakSolution}. The solution lies in $C([0,\infty),L^1_{\loc}(\R))$ and is, up to a null set, uniquely characterised by
    \begin{equation}\label{eq:forwards-solution-formula}
        u(X_t(x),t)=u_0(x),
    \end{equation}
where $X_t$ is the flow generated by the velocity field $b$.
\end{theorem}
Section~\ref{sec:multi-d} is devoted to the multi-dimensional setting and the corresponding result is summarized below (see Theorem~\ref{thm:uniquenessx-and-stability-multi-d} for a more precise formulation).
\begin{theorem}[Main Theorem, multiple dimensions]
If $b=b(x,t)\from\R^d\times[0,T]\to\R^d$ is bounded and log-Lipschitz continuous in $x$, and if $u_0$ is locally H\"older continuous, then there exists a unique locally Hölder continuous solution of \eqref{eq:transport}, and it is given by \eqref{eq:forwards-solution-formula}.
\end{theorem}

We also mention here that in Section~\ref{sec:backwards1d} we treat the one-dimensional backwards problem, showing that the ``canonical solution'' is a weak solution, and that it is unique in being stable with respect to smooth perturbations of $b$. In Section~\ref{sec:vanishing-viscosity} we study the viscous approximation of the multi-dimensional backwards equation by means of the associated stochastic differential equations, and prove that our solution is the vanishing viscosity solution. Finally, in Section~\ref{sec:inhomogeneous} we show that Duhamel's principle holds for the one-dimensional forwards equation.

\begin{remark}[Regarding the assumptions on $b$ and $u_0$]\label{rem: remarkOnTheAssumptions}\hspace{0pt}
\begin{enumerate}[label=(\roman*)]
\item Some assumptions can be weakened to local variants without much added difficulty: Assumption \ref{ass:bounded} can be replaced by $|b(x,t)|\leq m(t)(1+|x|)$ with $m\in L^1_{\loc}([0,\infty))$ and in assumption \ref{ass:osoc} and \ref{ass:oshc} one could replace the global modulus $\omega_b$ with a family of local ones $\{\omega_R(\cdot)\}_{R>0}$, each member valid for $x,y\in[-R,R]$. For the sake of simplicity, we instead make global assumptions.
 \item The continuity assumption \ref{ass:cont} on $b$ is to ensure that the Riemann–Stieltjes integral is well-defined, as one would otherwise need to deal with overlapping discontinuities (consider the classical example $b(x)=-\sign(x)$). With a correct interpretation of the integral it might be possible to extend the theory to discontinuous $b$ (as done for instance in Bouchut--James~\cite{BouJam97}), but this is beyond the scope of this paper.
\item  Assumption \ref{ass:osoc} ensures the well-posedness of the ODE \eqref{eq:ode}, while assumption \ref{ass:oshc} ensures the well-posedness of the Riemann–Stieltjes integral in \eqref{eq: theWeakRiemannStieltjesFormOfTheEquation}. Both \ref{ass:osoc} and \ref{ass:oshc} hold automatically when $b$ is one-sided Lipschitz continuous. Note also that \ref{ass:oshc} is not implied by \ref{ass:osoc}: The linear interpolation of $(x_n,\log(e+\frac{1}{x_n})^{-1})_{n\in \N}$ results in a concave Osgood modulus (provided $x_n\to0$ sufficiently fast) that fails to satisfy~\eqref{eq: oneSidedHölderContinuousOfAllOrdersLessThanOne}.
\item The regularity assumption \ref{ass:regularityOfInitialData} on $u_0$ is in particular implied by local Hölder continuity or, more generally, by local one-sided Hölder continuity; see Lemma \ref{lem:oneSidedHolderContinuityImpliesFiniteWeakerVariation}.
\end{enumerate}
\end{remark}

\section{Preliminaries}

\subsection{Notation}\label{sec:notation}

The set $B_r(x)$ will denote the open ball with radius $r$ centred at $x$.

Let $\rho\in C_c^\infty(\R)$ be a non-negative mollifier and let $\rho_\eps(x)\coloneqq \eps^{-1}\rho(\eps^{-1}x)$. For a function $v$ we let $[v]^\eps$ denote the convolution (mollification) by $\rho_\eps$ in the $x$ variable. Likewise, we let $[v]^{\eps,\delta}$ denote convolution with $\rho_\eps$ in the $x$-variable and with $\rho_\delta$ in the $t$-variable. In $d$ space dimensions (cf.~Section~\ref{sec:multi-d}), the symmetric mollifier $x\mapsto\eps^{-d}\rho(\eps^{-1}|x|)$ is used.

\subsection{ODEs with Osgood velocity fields}\label{sec: ODEsWithOsgoodVelocityFields}
In this section we review some well-posedness theory on the ordinary differential equation \eqref{eq:ode}. By a \emph{solution of \eqref{eq:ode}} we will mean a Lipschitz continuous function $t\mapsto X_t$ such that \eqref{eq:ode} is satisfied at almost every $t>0$, and such that $X_0=x$.

The following lemma is often attributed to Bihari~\cite{Bihari1956} and LaSalle~\cite{LaSalle1949}, but is originally due to Osgood~\cite{Osg1898}.

\begin{lemma}[Osgood's inequality]\label{lem:osgood}
Let $y\from[0,\infty)\to[0,\infty)$ satisfy the differential inequality
\begin{equation}\label{eq:osgood-inequality}
y(t) \leq y_0 + \int_0^t \lambda(s)\omega(y(s))\,ds
\end{equation}
where $y_0\geq0$, $\lambda\in L^1_\loc([0,\infty))$ is non-negative, and $\omega$ satisfies \eqref{eq: OsgoodCriterion}.
Then $y(t) \leq \Psi(y_0,t)$, where $\Psi(y_0,t)\coloneqq G_{y_0}^{-1}\bigl(\Lambda(t)\bigr)$, and $G_{y_0}$ and $\Lambda$ are the functions $G_{y_0}(y)\coloneqq \int_{y_0}^y \frac{1}{\omega(s)}\,ds$ and $\Lambda(t)\coloneqq \int_0^t \lambda(s)\,ds$.
\end{lemma}

\begin{theorem}\label{thm:ode-well-posed}
Assume that $b$ satisfies Assumptions~\ref{ass:bounded},~\ref{ass:cont},~\ref{ass:osoc}. Then the ODE \eqref{eq:ode} with initial data $X_0=x$ has exactly one solution, which we denote by $t\mapsto X_t(x,0)$ (for $t\geq 0$). More generally, the flow starting at $(x,s)$ is denoted $t\mapsto X_t(x,s)$ (for $t\geq s$). This flow is continuous and satisfies
\begin{equation}\label{eq:flowstability}
    |X_t(x,s)-X_t(y,s)| \leq \Psi\bigl(|x-y|,t-s\bigr) \qquad\forall\ x,y\in\R,\ t\geq s,
\end{equation}
where $\Psi$ is as in Lemma \ref{lem:osgood}. The flow map $x\mapsto X_t(x,s)$ is surjective and increasing for fixed $0\leq s \leq t$. If $X_\eps$ is the flow corresponding to a mollified velocity $[b]^\eps$, then $X_\eps\to X$ uniformly on compacts as $\eps\to0$.
\end{theorem}
\begin{proof}
Without loss of generality we assume that $s=0$. The existence of solutions follows from the fact that $b$ is continuous in $x$. Let $X_t, Y_t$ be solutions of \eqref{eq:ode} with initial data $x$ and $y$, respectively. Then
\[
\frac{d}{dt}|X_t-X_t| \leq \lambda(t)\omega\bigl(|X_t-Y_t|\bigr).
\]
Integrating with respect to $t$, we see that $y(t)\coloneqq |X_t-Y_t|$ satisfies \eqref{eq:osgood-inequality}, whence \eqref{eq:flowstability} follows from Lemma~\ref{lem:osgood}. The fact that $X_\eps\to X$ is a standard compactness argument: The family $(t\mapsto X_{\eps,t}(x))_{\eps>0}$ is uniformly Lipschitz, so it converges along a subsequence; the limit is a solution since $b$ is continuous; and the limit is unique, so the entire sequence converges; and since the estimate \eqref{eq:flowstability} is uniform in $\eps>0$, the convergence is uniform in both $x$ and $t$.
\end{proof}

\subsection{Functions of finite $p$-variation}\label{sec:p-variation}
\newcommand{\bx}{{\mathbf{x}}}
\newcommand{\by}{{\mathbf{y}}}
\newcommand{\bz}{{\mathbf{z}}}
\newcommand{\Part}{{\mathcal{P}}}

In this section we review theory of functions of finite $p$-variation. The presentation is mostly based on Young's classical reference~\cite{Young1936}. We refer to Friz and Victoir~\cite[Chapter~5]{FriVic2010} for an exposition from a more modern, rough paths viewpoint.

\begin{definition}\label{def:partition}
Let $A\subseteq\R$ be an interval. A \emph{partition} of $A$ is a tuple $\bx=(x_0,\dots,x_N)$ where $x_0,\dots,x_N\in A$ and $x_0<\cdots<x_N$. The set of all partitions of $A$ is denoted $\Part(A)$. When $A$ is bounded, the \emph{resolution} of a partition $\bx$ is the number
\[
|\bx|\coloneqq \max\bigl\{x_0-\inf A, x_1-x_0,\dots,x_N-x_{N-1}, \sup A-x_N\bigr\}.
\]
We write $\by \prec \bx$ if $\mathbf{x}=(x_0,\dots,x_N)$, $\mathbf{y}=(y_0,\dots,y_{N-1})$ and $y_i\in[x_i,x_{i+1}]$ for all $i=0,\dots,N-1$.
\end{definition}

\begin{definition}
Consider an interval $A\subseteq \R$ and let $p\in(0,\infty]$. We define the $p$-variation of a function $f\colon A\to \R$ by
\[
|f|_{V^p(A)} \coloneqq
\begin{cases}
    \displaystyle \sup_{\bx\in\Part(A)}\Bigg(\sum_{i=1}^{n}|f(x_{i})-f(x_{i-1})|^p\Bigg)^{1/p} &\text{for } p\in(0,\infty),\\
     \displaystyle\sup_{x,y\in  A}|f(x)-f(y)| & \text{for } p=\infty.
\end{cases}
\]
We denote $V^p(A)\coloneqq\{f\from A\to\R : |f|_{V^p(A)}<\infty\}$ for $p\in(0,\infty]$. When the domain $A$ is apparent from the context, we write simply $|f|_{V^p}$.

For $\delta>0$ we also define
\[
|f|_{V^\infty_\delta(A)} \coloneqq \sup_{\substack{x,y\in  A\\|x-y|<\delta}}|f(x)-f(y)|.
\]
Note that $|f|_{V^\infty_\delta(A)}\to0$ as $\delta\to0$ if and only if $f$ is uniformly continuous on $A$.
\end{definition}

We list some immediate properties:
\begin{proposition}\label{prop:p-variation}
Let $A\subseteq \R$ be an interval and $f\colon A\to \R$. Then:
\begin{enumerate}[label=\textit{(\roman*)}]
    \item\label{prop:p-variation-limits}
    If $|f|_{V^p(A)}<\infty$ for some $p<\infty$, then $f$ has left and right limits in $A$,
    \item\label{prop:p-variation-monotone}
    $|f|_{V^{q}(A)}\leq |f|_{V^{p}(A)}$ whenever $p\leq q$,
    \item\label{prop:p-variation-seminorm}
    $|\cdot|_{V^{p}(A)}$ is a seminorm for $p\in[1,\infty]$,
    \item\label{prop:p-variation-TV}
    $|f|_{V^{1}(A)} = |f|_{TV(A)}$, where the latter is the total variation seminorm,
    \item\label{prop:p-variation-holder}
    $|f|_{V^{p}(A)}\leq|A| |f|_{C^{0,1/p}(A)}$ for $p\in[1,\infty)$, where $|A|$ is the length of the interval and $|\cdot|_{C^{0,1/p}(A)}$ is the $1/p$-Hölder seminorm.
    \item\label{prop:p-variation-product-rule}
    $|fg|_{V^p(A)} \leq |f|_{V^p(A)}\|g\|_{L^\infty(A)} + \|f\|_{L^\infty(A)}|g|_{V^p(A)}$ for any $g\from A\to\R$.
\end{enumerate}
\end{proposition}
We extend this pointwise-dependent concept to equivalence classes $f\in L^\infty(A)$ by setting $|f|_{V^p(A)}=\inf\bigl\{|g|_{V^p(A)} : g=f \text{ a.e.}\bigr\}$. From \ref{prop:p-variation-limits} above we see that if $f\in L^\infty(A)$ is such that $|f|_{V^p(A)}<\infty$ with $p<\infty$, then $f$ admits a right-continuous representation $\overline{f}$. In these cases we will not distinguish between the equivalence class $f$ and the function $\overline{f}$; it can be easily checked that $|f|_{V^p(A)}=|\overline{f}|_{V^p(A)}$.

By \ref{prop:p-variation-holder} above, $1/p$-Hölder continuity implies locally finite $p$-variation when $p\in[1,\infty)$.
One may intuitively expect something similar to hold for \textit{one-sided} $1/p$-Hölder regularity, as this regularity at least controls all increasing increments, which accounts for ``half the variation''.
However, a one-sided version of \ref{prop:p-variation-holder} holds true only for the trivial case $p=1$ (cf.~Remark~\ref{rem: theCounterExampleWithInfinitePVariation}). Instead we have the weaker, but useful, result for general $p\in(1,\infty)$:

\begin{lemma}\label{lem:oneSidedHolderContinuityImpliesFiniteWeakerVariation}
Let $A=[a_0,a_1]$ be a bounded and closed interval, $p\in(1,\infty)$, and $f\colon A\to \R$ be one-sided $1/p$-Hölder continuous, that is,
\begin{align*}
    C_p \coloneqq \sup_{a_0\leq x<y\leq a_1}\frac{f(y)-f(x)}{(y-x)^{1/p}}<\infty.
\end{align*}
Then the $q$-variation of $f$ is finite for all $q\in(p,\infty]$. Specifically, we have bounds
\begin{subequations}
\begin{empheq}[left={\empheqlbrace}]{alignat=1}
     |f|_{V^q(A)} &\leq c_1|f|_{V^\infty(A)} + c_2 C_p(a_1-a_0)^{1/p}  \qquad\text{for } q\in(p,\infty), \label{eq: explicitBoundOnTheQVariation} \\
    |f|_{V^\infty(A)} &\leq f(a_0)-f(a_1)+2C_p(a_1-a_0)^{1/p} \label{eq: explicitBoundOnTheInftyVariation}
\end{empheq}
\end{subequations}
for coefficients $c_1=c_1(p,q)$ and $c_2=c_2(p,q)$ depending only on $p$ and $q$.
\end{lemma}
\begin{proof}
Note first that $f$ is necessarily bounded since
\begin{align*}
    \sup_{x\in A}f(x)\leq f(a_0) + C_p(a_1-a_0)^{\frac{1}{p}},\qquad \inf_{x\in A}f(x)\geq f(a_1) - C_p(a_1-a_0)^{\frac{1}{p}}.
\end{align*}
In particular, $|f|_{V^\infty(A)}$ satisfies \eqref{eq: explicitBoundOnTheInftyVariation}.
For the remaining bound, fix $q\in(p,\infty)$, let $\bx=(x_0,\cdots,x_m)$ be a partition of $ A$, and denote $\Delta_i x = x_i-x_{i-1}$ and $\Delta_i f = f(x_i)-f(x_{i-1})$.
As the partition is arbitrary, it suffices to prove that $\|\Delta_{\subscriptcdot} f\|_{\ell^q}$ is bounded by an expression like that of the right-hand side of \eqref{eq: explicitBoundOnTheQVariation}.
For this, introduce the index sets $I^+ = \{i\in\{1,\dots,m\} : \Delta_i f\geq 0\}$ and $I^- = \{i\in\{1,\dots,m\} : \Delta_i f< 0\}$, and note that
\begin{equation}\label{eq: TheSplitSumUsingTheIndexSets}
    \Bigg(\sum_{i=1}^m |\Delta_i f|^q\Bigg)^{\frac{1}{q}} \leq \Bigg(\sum_{i\in I^+} (\Delta_i f)^q\Bigg)^{\frac{1}{q}} + \Bigg(\sum_{i\in I^{-}} (-\Delta_i f)^q\Bigg)^{\frac{1}{q}}.
\end{equation}
The first part on the right-hand side of \eqref{eq: TheSplitSumUsingTheIndexSets}, is easily bounded:
\begin{align*}
\bigg(\sum_{i\in I^+} (\Delta_i f)^q\bigg)^{\frac{1}{q}}
&\leq \bigg(\sum_{i\in I^+} (\Delta_i f)^p\bigg)^{\frac{1}{p}}
\leq C_p\bigg(\sum_{i\in I^+}^m (\Delta_i x)\bigg)^{\frac{1}{p}}\leq C_p(a_1-a_0)^\frac{1}{p}.
\end{align*}

Dealing with the last term in \eqref{eq: TheSplitSumUsingTheIndexSets} is less straight-forwards. Partition $I^-$ into subsets $I^-_0, I^-_1, \dots$, where
\[
I^-_n = \Big\{i\in I^- : \frac{1}{2^{n+1}}< \frac{-\Delta_i f}{|f|_{V^\infty}} \leq \frac{1}{2^{n}} \Big\}, \qquad n=0,1,\dots
\]
Fix $n$. To avoid sub-subindices, we relabel the points $x_{i-1},x_i$ for $i\in I_n^-$ as $y_1< z_1\leq y_2< z_2\leq\dots\leq y_k< z_k$ where $k\coloneqq\#I_n^-$; that is, for every $i\in I^-_n$ there is a corresponding $j\in\{1,\dots, k\}$ such that $\Delta_i f=f(z_j)-f(y_j)$. We will need to control $k$, which can be done using the identity
\begin{equation*}\label{eq: telescopingSum}
 L\coloneqq \sum_{j=1}^k \bigl(f(y_j)-f(z_j)\bigr) =   \big(f(z_k)-f(y_1)
\big) + \sum_{j=2}^k \bigl(f(y_j)-f(z_{j-1})\bigr)\eqqcolon R_1+R_2,
\end{equation*}
and the three bounds
\begin{equation*}
    \begin{split}
         L=\,\sum_{i\in I^-_n} -\Delta_i f\geq \frac{|f|_{V^{\infty}}}{2^{n+1}}k, \qquad R_1\leq |f|_{V^\infty},\quad \\ R_2\leq\, C_p\sum_{j=2}^k (y_j-z_{j-1})^{\frac{1}{p}}\leq C_p (a_1-a_0)^{\frac{1}{p}}k^{1-\frac{1}{p}},
    \end{split}
\end{equation*}
where the latter bound for $R_2$ follows by H\"older's inequality.
Combining the identity and the bounds, and multiplying by $2^{n+1}/|f|_{V^\infty}$, yields
\begin{align*}
    \#I_n^- = k\leq 2^{n+1}  + 2^{n}\Biggl(\underbrace{\frac{2C_p (a_1-a_0)^{\frac{1}{p}}}{|f|_{V^\infty}}}_{\eqqcolon\,M}\Biggr)k^{1-\frac{1}{p}}.
\end{align*}
By Young's inequality, $2^nMk^{1-\frac{1}{p}}\leq \frac{1}{p}2^{np}M^{p} + (1-\frac{1}{p})k$, and rearranging, we find that
\begin{align*}
     \#I_n^-\leq p2^{n+1}  + 2^{np}M^p.
\end{align*}
Hence,
\begin{align*}
     \sum_{i\in I^-_n} |\Delta_i f|^q\leq  \bigg(\sup_{i\in I^-_n}|\Delta_i f|^q \biggr)\cdot \big(\#I_n^-\big)\leq  |f|_{V^\infty}^q\bigg(\frac{2p}{2^{n(q-1)}}  + \frac{M^p}{2^{n(q-p)}}\bigg).
\end{align*}
This latter expression is summable over $n\in \{0,1,\dots\}$; we conclude that
\begin{align*}
\Bigg(\sum_{i\in I^-} |\Delta_i f|^q\Biggr)^{\frac{1}{q}}
&= \Bigg(\sum_{n=0}^\infty\sum_{i\in I^-_n} |\Delta_i f|^q \Biggr)^{\frac{1}{q}}\leq |f|_{V^\infty}\bigg(\frac{2p}{1-2^{1-q}}  + \frac{M^p}{1-2^{p-q}}\bigg)^{\frac{1}{q}}.
\end{align*}
Inserting this, and the bound for the sum over $I^+$, in \eqref{eq: TheSplitSumUsingTheIndexSets}, followed by some simplifications (exploiting the sub-additivity of $x\mapsto x^{\frac{1}{q}}$ and Young's inequality) we get an expression like the one from the right-hand side of \eqref{eq: explicitBoundOnTheQVariation}, and so we are done.
\end{proof}

\begin{remark}\label{rem: theCounterExampleWithInfinitePVariation}
    As can be seen from the previous proof, the coefficients $c_1$ and $c_2$ blow up as $q\downarrow p$.
    This is unavoidable: For every $p\in(1,\infty)$, there are functions $f\colon (0,1)\to \R$ that are one-sided $\frac{1}{p}$-Hölder continuous but with $|f|_{V^p}=\infty$.
    An example of such a function is $f(x)=\sum_{k=1}^\infty b_k(x)2^{-k/p}$, where $b_k(x)$ denotes the $k$'th digit in the binary expansion of $x\in(0,1)$.
    Since this example will play no role in the paper, we omit the (somewhat lengthy) proof of the aforementioned properties of $f$.
\end{remark}
\begin{corollary}[Finite variation of the velocity field]\label{cor: finitePVariationOfTheVelocityB}
    The velocity field $b$ has locally finite $q$-variation for all $q\in(1,\infty)$ in the following sense: With $\lambda$ and $\omega_b$ as in Assumptions~\ref{ass:osoc} and \ref{ass:oshc} we have
    \begin{equation*}
        |b(t)|_{V^q([-R,R])}\lesssim\big(1+\lambda(t)\big) \qquad\text{for all }R>0.
    \end{equation*}
The implicit constant depends only on $q\in(1,\infty)$, $R\in (0,\infty)$, $\|b\|_{L^\infty}$, and $\omega_b$.
\end{corollary}
\begin{proof}
    For $q\in (1,\infty)$, we can pick $p\in(1,q)$ and use Assumptions~\ref{ass:osoc} and \ref{ass:oshc} to find a $C_p$ such that $b(x+h,t)-b(x,t)\leq C_p \lambda(t)h^{1/p}$,
    for all $h>0$ and $x\in\R$ and a.e.~$t>0$. We then get the desired bound from Lemma \ref{lem:oneSidedHolderContinuityImpliesFiniteWeakerVariation}.
\end{proof}

The following lemma will be used in the renormalization argument in Theorem~\ref{thm:renormalizability}.
\begin{lemma}\label{lem:l1-translation-error-estimate}
Let $A=(a_0,a_1)\subseteq\R$,  $p\in[1,\infty)$, $v\in V^p(A)$, and $|y|\leq \frac{a_1-a_0}{2}$. Then
\begin{equation}\label{eq:simple-translation-estimate}
\int_{a_0+|y|}^{a_1-|y|}|v(x+y)-v(x)|^p\,dx \leq  |y||v|_{V^p(A)}^p.
\end{equation}
Consequently, if $v\in V^p(\R)$, $w\in V^q(\R)$ with $1/\theta\coloneqq 1/p+1/q>1$ then
\begin{equation}\label{eq:double-translation-estimate}
\int_{\R} |v(x+\delta)-v(x)||w(x+\delta)-w(x)|\,dx \leq \delta |v|_{V_\delta^\infty}^{1-\theta}|w|_{V_\delta^\infty}^{1-\theta} |v|_{V^p}^\theta |w|_{V^q}^\theta.
\end{equation}
\end{lemma}
\begin{proof}
Without loss of generality let $y> 0$ and assume $|a_0|,|a_1|<\infty$ (the unbounded cases follow by a limit argument). Write $x\in\R$ as $x=ky + z$ for $k\in\Z$ and $z\in[0,y)$, and let $k_0,k_1\in \Z$ be such that $(a_0+y,a_1-y)\subseteq (k_0 y, k_1 y)\subseteq (a_0,a_1)$. We get
\begin{equation*}
   \int_{a_0+y}^{a_1-y}|v(x+y)-v(x)|^p\,dx \leq \sum_{k=k_0}^{k_1}\int_0^y \big|v\big((k+1)y+z\big) - v\big(ky + z\big)|^p\,dz\leq y |v|_{V ^p(A)}^p
\end{equation*}
where the inequality follows from interchanging the sum and the integral. To prove \eqref{eq:double-translation-estimate}, we first note that $\frac{\theta}{p}+\frac{\theta}{q}=1$. By Hölder's inequality we then have
\begin{equation}\label{eq:midwayInProofToDoubleTransEstimate}
    \begin{split}
            &\int_{\R} |v(x+\delta)-v(x)||w(x+\delta)-w(x)|\,dx\\ &\qquad\leq \bigg(\int_{\R} |v(x+\delta)-v(x)|^{\frac{p}{\theta}}\, dx\bigg)^{\frac{\theta}{p}}\bigg(\int_{\R}|w(x+\delta)-w(x)|^{\frac{q}{\theta}}\,dx\bigg)^{\frac{\theta}{q}}.
    \end{split}
\end{equation}
Further, using \eqref{eq:simple-translation-estimate} we get
\begin{align*}
    \int_{\R} |v(x+\delta)-v(x)|^{\frac{p}{\theta}}\, dx\leq |v|_{V^\infty_\delta}^{\frac{p}{\theta}(1-\theta)}\int_{\R}|v(x+\delta)-v(x)|^p\,dx\leq \delta|v|_{V^\infty_\delta}^{\frac{p}{\theta}(1-\theta)}|v|_{V^p}^p.
\end{align*}
With a similar estimate for $w$, we get \eqref{eq:double-translation-estimate} from \eqref{eq:midwayInProofToDoubleTransEstimate}.
\end{proof}

\subsection{Riemann--Stieltjes integration}\label{sec:riemann-stieltjes}

In this section we review the theory of Riemann--Stieltjes integrals. As in the previous section, the classical reference is Young~\cite{Young1936}, while an exposition from the viewpoint of rough paths is found in~\cite[Chapter~6]{FriVic2010}.

\begin{definition}
Let $f,g\from A\to\R$ be given, where $A\subset \R$ is a bounded interval. The Riemann--Stieltjes integral
\begin{equation}\label{eq: riemannStieltesIntegrationOnDisplay}
\int_A f(x)\,d_xg(x)
\end{equation}
is given by the limit (whenever it exists) of the Riemann--Stieltjes sums
\begin{equation}\label{eq:riemann-stieltjes-sum}
I_{\bx,\by}(f,g) \coloneqq \sum_{k=1}^N f(y_k)\bigl(g(x_k)-g(x_{k-1})\bigr)
\end{equation}
as the mesh size $|\bx|$ tends to zero, where $\bx,\by\in\Part(A)$ are arbitrary partitions with $\by\prec\bx$ (cf.~Definition~\ref{def:partition}). When $A\subseteq\R$ is unbounded, the integral is defined as the $R\uparrow\infty$-limit when integrating over $(-R,R)\cap A$.
\end{definition}

Each of the results in the following theorem is either contained in, or easily obtained from, the classical reference \cite{Young1936}.

\begin{theorem}\label{thm:young-integral-properties}
Let $p,q\geq 1$ be real numbers such that $\frac{1}{p}+\frac{1}{q} > 1$, and consider an interval $A=(a_0,a_1)$. Let $f\in V^p(A)$ and $g\in V^q(A)$, and with either $f$ or $g$ continuous. We then have the following:
\begin{enumerate}[label=\textit{(\roman*)}]
\item\label{thm:young-integral-properties-well-defined}
Well-posedness of the integral: The integrals $\int_{A} f\,d_xg$ and $\int_{A} g\,d_xf$ are well-defined and bilinear in $(f,g)$.
\item\label{thm:young-integral-properties-integration-by-parts}
Integration by parts formula: We have
\[
\int_{A} f\,d_xg + \int_{A} g\,d_xf = f(a_1-)g(a_1-)-f(a_0+)g(a_0+)
\]
where $h(a_1-)$ and $h(a_0+)$ denotes the left and right limits of $h=f,g$ at $a_1$ and $a_0$, respectively.
\item \label{thm:young-integral-properties_product-rule}Product rule: If $h\in V^p(A)\cap V^q(A)$ then
\[
\int_{A} f\,d_x(gh) = \int_{A} fg\,d_xh + \int_{A} fh\,d_xg.
\]
\item \label{thm:young-integral-properties_size-control} Size control: For any $\theta\in[0,\theta_0)$, where $\theta_0\coloneqq q\bigl(\tfrac1p+\tfrac1q-1\bigr)\leq1$, we have
\[
\biggl|\int_{A} f\,d_xg\biggr| \leq \|f\|_{L^\infty} |g|_{V^\infty} + C_{p,q,\theta} |f|_{V^p} |g|_{V^q}^{1-\theta}|g|_{V^\infty}^{\theta}
\]
for some $C_{p,q,\theta}>0$.
\item\label{thm:young-integral-properties-conv-rate}
Convergence rate: Let $\bx,\by\in\Part(A)$ with $\by\prec\bx$, let $\delta\coloneqq |\bx|$, and let $\theta$ be as in \ref{thm:young-integral-properties_size-control}. Then the error in the Riemann--Stieltjes sum \eqref{eq:riemann-stieltjes-sum} can be bounded by
\[
\Biggl|\int_{A} f(x)\,d_xg(x) - \sum_{k=1}^N f(y_k)\bigl(g(x_k)-g(x_{k-1})\bigr) \Biggr| \leq
C_{p,q,\theta} |f|_{V^p} |g|_{V^q}^{1-\theta}|g|_{V^\infty_\delta}^\theta .
\]


\item \label{thm:young-integral-properties_pw-convergence}
Stability: Let $f_1,f_2,\ldots$ and $g_1,g_2,\ldots$ have uniformly bounded $p$- and $q$-variation, respectively. If $f_n\to f$ pointwise on a dense set in $A$ and $g_n\to g\in C(A)$ uniformly, then $\int_{A} f_n\,d_xg_n \to \int_{A} f\,d_xg$ as $n\to\infty$.
\end{enumerate}
\end{theorem}


We will also need to deal with iterated integrals, where the integral in one of the directions is a Riemann--Stieltjes integral:
\begin{lemma}\label{lem:iterated-integral}
Let $B\subset\R^m$ (for $m\in\N$) be measurable, and let $A\subseteq \R$ be an interval. Let further $f,g\from A\times B\to\R$ be measurable functions. Let $p,q\geq 1$ satisfy $1/p+1/q>1$, and assume that for a.e.~$y\in B$,
\begin{enumerate}[label=(\roman*)]
\item $|f(\cdot,y)|_{V^p(A)},|g(\cdot,y)|_{V^q(A)}<\infty$,
\item $g(\cdot,y)$ is continuous.
\end{enumerate}
Then the map $y \mapsto \int_A f(x,y)\,d_xg(x,y)$ is measurable. 
\end{lemma}

\begin{proof}
As explained after Proposition \ref{prop:p-variation}, we can for a.e.~$y\in B$ identify $f(\cdot,y)$ with its right-continuous version. Note that this version, as a function in $(x,y)$, differs from the original one on at most a null-set in $A\times B$ since it coincides with $\lim_{\varepsilon\downarrow 0} \eps^{-1}\int_0^\varepsilon f(x+z,y)\, dz$ by the Lebesgue differentiation theorem on sections.
It then follows that for all $x\in A$, the $x$-sections $y\mapsto f(x,y)$ and $y\mapsto g(x,y)$ are measurable since they are the pointwise limits (as $\varepsilon\downarrow 0$) of the measurable functions $y\mapsto \eps^{-1}\int_0^\varepsilon f(x+z,y)\, dz$ and $y\mapsto \eps^{-1}\int_0^\varepsilon g(x+z,y)\, dz$. Hence, for every pair of partitions $\bz\prec\bx$, the corresponding Riemann--Stieltjes sum $I_{\bx,\bz}\bigl(f(\cdot,y),g(\cdot,y)\bigr)$ (cf.~\eqref{eq:riemann-stieltjes-sum})
is measurable in $y$. Therefore, being the pointwise a.e.~limit of measurable functions, the limit $\int_A f(x,y)\,d_xg(x,y)$ is $y$-measurable.
\end{proof}

\begin{lemma}[``Fubini's theorem'']\label{lem:fubini}
Let $A\subseteq \R$ be an interval, $B\subset\R^m$ (for $m\in\N$) be measurable and bounded, and let $p,q\geq 1$ satisfy $1/p+1/q>1$. Let $f\from A\times B\to\R$ be bounded and measurable and satisfy $\esssup_{y\in B} |f(\cdot,y)|_{V^p}<\infty$, and let $g\in C(A)\cap V^q(A)$.
Then the map $y \mapsto \int_{A} f(x,y)\,d_xg(x)$ is Lebesgue integrable, and
\[
\int_B\int_{A} f(x,y)\,d_xg(x)\,dy = \int_{A} \int_B f(x,y)\,dy\,d_xg(x).
\]
\end{lemma}

\begin{proof}
Denote $I(f(y),g)\coloneqq\int_{A} f(x,y)\,d_xg(x)$, and for a pair of partitions $\bz\prec\bx$, let $I_{\bx,\bz}(f(y),g)$ denote the Riemann--Stieltjes sum \eqref{eq:riemann-stieltjes-sum}.
Fix some $\theta\in(0,\theta_0)$, where $\theta_0\coloneqq q\bigl(\frac1p+\frac1q-1\bigr)$ and let $\delta\coloneqq|\bx|$. Then for a.e.~$y\in B$, Theorem~\ref{thm:young-integral-properties}~\ref{thm:young-integral-properties-conv-rate} yields
\[
\Bigl|I(f(y),g) - I_{\bx,\bz}(f(y),g)\Bigr| \leq
C_{p,q,\theta} |f(y)|_{V^p} |g|_{V^q}^{1-\theta} |g|_{V^\infty_\delta}^\theta,
\]
which vanishes as $\delta\to0$, since $g$ is uniformly continuous. By Theorem \ref{thm:young-integral-properties} \ref{thm:young-integral-properties_size-control}, we have
\[
\bigl|I(f(y),g)\bigr| \leq 2\|f\|_{L^\infty}\|g\|_{L^\infty} + C_{p,q,\theta} |f(y)|_{V^p} |g|_{V^q}^{1-\theta} |g|_{V^\infty_\delta}^\theta,
\]
which is uniformly bounded, by our assumptions on $f$ and $g$. Hence, $y\mapsto I(f(y),g)$ is the uniform limit of uniformly bounded functions, so Lebesgue's dominated convergence theorem yields
\begin{align*}
&\int_B\int_{A} f(x,y)\,d_xg(x)\,dy =
\int_B I(f(y),g)\,dy = \int_B \lim_{\substack{|\bx|\to0\\ \bz\prec\bx}}I_{\bx,\bz}(f(y),g)\,dy \\
&\quad= \lim_{\substack{|\bx|\to0\\ \bz\prec\bx}} \int_B I_{\bx,\bz}(f(y),g)\,dy
= \lim_{\substack{|\bx|\to0\\ \bz\prec\bx}} \sum_{k=1}^N\int_B f(z_k,y)\,dy\, \bigl(g(x_k)-g(x_{k-1})\bigr) \\
&\quad= \int_{A} \int_B f(x,y)\,dy\,d_xg(x).\qedhere
\end{align*}
\end{proof}

\section{$L^1$ estimates on Riemann--Stieltjes integrals}\label{sec:nonlinear-l1-estimate}

This section is devoted to developing bounds on the transport term in \eqref{eq: theWeakRiemannStieltjesFormOfTheEquation}.
As can be motivated by the discussion in Section \ref{sec: sketchOfTheUniquenessArgument}, and specifically equation \eqref{eq: introductionThePreGeneralizedGrönwallInequality}, we aim to bound a Riemann--Stieltjes integral $\int_{\R}\gamma \, d_x\beta$ (for some non-negative $\gamma$) in terms of $\|\gamma\|_{L^1(\R)}$. In Section~\ref{sec:uniqueness-and-stability} we will apply this to $\gamma= |u|(t)\phi(t)$ and $\beta=b(t)$.

The following is the main result of this section.

\begin{theorem}[Controlling the transport term with $p$-variation]\label{thm: controlOnIntegralWithPVariationandOnesidedRegularity}
    Let $p\in[1,\infty)$, let $\gamma\in V^p(\R)$ be non-negative and have support in an interval of length $L>0$, and let $\beta\in C^0(\R)$ admit a concave one-sided modulus of continuity $\omega$ satisfying Assumption \ref{ass:oshc}. Then
      \begin{equation}\label{eq: controlOnIntegralWithPVariationandOnesidedRegularity}
       \int_{\R}\gamma(x)\, d_x\beta(x)\leq |\gamma|_{V^p}L^{1-\frac{1}{p}}\omega^*\bigg(\frac{\|\gamma\|_{L^1}}{|\gamma|_{V^p}L^{1-\frac{1}{p}}}\bigg),
    \end{equation}
    where $\omega^*(h)\coloneqq h^{2-2p} \omega\bigl( h^{2p-1}\bigr) + c_ph$, and where $c_p\geq 0$ only depends on $p$.
\end{theorem}
\begin{remark}
Note that $\omega^*$ need not be concave nor strictly increasing (this is a consequence of not truly optimizing in $r_0$ in the proof). Still, observe that $C\mapsto C\omega^*(h/C)$ is non-decreasing (by concavity of $\omega$) and that $\omega^*$ is an Osgood modulus (i.e., satisfies \eqref{eq: OsgoodCriterion}) whenever $\omega$ is. Indeed,
\[
\int_0^\eps\frac{dh}{\omega^*(h)}=\int_0^\eps \frac{h^{2-2p}dh}{\omega(h^{2p-1}) + c_ph^{2p-1}}
= \frac{1}{2p-1}\int_0^{\eps^{2p-1}}\frac{dz}{\omega(z)} = \infty
\]
for any $\eps>0$.
 \end{remark}

The rest of this section is devoted to proving Theorem~\ref{thm: controlOnIntegralWithPVariationandOnesidedRegularity}.
The special case $p=1$ --- namely, when $\gamma$ has bounded total variation --- turns out to be rather elementary, and is proved below in Section~\ref{sec:controlOnIntegral-TV}.
The improvement from $p=1$ to $p>1$ requires new techniques. In Section~\ref{sec: lowerBoxEnvelope} we show that any non-negative $\gamma\in V^p(\R)$ can be decomposed into a sum of non-negative functions of bounded variation (Lemma~\ref{lem: monotoneDecomposition}) in a controlled manner.
In Section~\ref{sec:controlling-transport-term} we apply the aforementioned decomposition to generalise to $p>1$.

\subsection{The case $p=1$}\label{sec:controlOnIntegral-TV}
In this section we prove the particular case of Theorem \ref{thm: controlOnIntegralWithPVariationandOnesidedRegularity} with $p=1$:

\begin{proposition}
\label{prop: controlOnIntegralWithBVandOnesidedRegularity}
    Let $0\leq \gamma\in BV(\R)$ and let $\beta\in C^0(\R)$ admit a concave one-sided modulus of continuity $\omega$. Then
    \begin{equation}\label{eq: controlOnIntegralWithBVandOnesidedRegularity}
       \int_{\R}\gamma\, d_x\beta\leq |\gamma|_{TV}\omega\bigg(\frac{\|\gamma\|_{L^1}}{|\gamma|_{TV}}\bigg).
    \end{equation}
\end{proposition}
\begin{remark}
    The right-hand side of \eqref{eq: controlOnIntegralWithBVandOnesidedRegularity} blows up as $|\gamma|_{TV}\to\infty$ unless $\omega$ is linear ($\beta$ is one-sided Lipschitz continuous), in which case the right-hand side reduces to $L\|\gamma\|_{L^1}$.
\end{remark}
\begin{proof}[Proof of Proposition \ref{prop: controlOnIntegralWithBVandOnesidedRegularity}]
    By a density argument, and stability of the Young integral (Theorem~\ref{thm:young-integral-properties}~\ref{thm:young-integral-properties_pw-convergence}), it suffices to consider $\beta\in C^\infty_c(\R)$. Define the superlevel set $E_\tau\coloneqq \{x : \gamma(x)>\tau\}$ and the two functions
\[
    g(\tau) \coloneqq \mathcal{H}^1\bigl(E_\tau\big), \qquad h(\tau)\coloneqq \mathcal{H}^0\bigl(\partial E_\tau\bigr),
\]
where $\mathcal{H}^1,\mathcal{H}^0$ denote the one- and zero-dimensional Hausdorff measure, respectively. Set $M\coloneqq \|\gamma\|_{L^\infty}$ and note that $g,h>0$ in $(0,M)$, while $g=h\equiv0$ in $(M,\infty)$. By the layer cake representation of $\gamma$ and the co-area formula we have the two identities
\begin{equation}\label{eq:epigraph-integrals}
\|\gamma\|_{L^1} = \int_0^M g(\tau)\,d\tau, \qquad |\gamma|_{TV} = \int_0^M h(\tau)\,d\tau.
\end{equation}
It follows that for almost every $\tau>0$ we have $h(\tau)<\infty$, so the set $E_\tau$ is a finite union of intervals. For these $\tau$, let $N_\tau$ denote the number of such intervals and let $x_k^\tau< y_k^\tau$, for $k=1,\dots, N_\tau$ be an indexing of the endpoints of these intervals; observe that
\begin{equation}\label{eq:epigraph-properties}
    h(\tau)=2N_\tau, \qquad g(\tau) = \sum_{k=1}^{N_\tau} y_k^\tau-x_k^\tau.
\end{equation}
We compute
\begin{align*}
    \int_{\R}\gamma(x)\, d_x\beta(x) &= \int_{\R}\gamma(x)\beta'(x)\, dx
    = \int_0^M\int_{\R}\mathbbm{1}(\gamma(x)>\tau)\beta'(x)\, dx\,d\tau\\
        &= \int_0^M\sum_{k=1}^{N_\tau}\big(\beta(y_k^\tau)-\beta(x_k^\tau)\big)\,d\tau
        \leq \int_0^M\sum_{k=1}^{N_\tau}\omega\bigl(y_k^\tau-x_k^\tau\bigr)\,d\tau.
\end{align*}
Dividing and multiplying this last sum by $N_\tau = h(\tau)/2$ followed by Jensen's inequality, we can bound it by
\begin{align*}
\int_{\R}\gamma(x)\, d_x\beta(x) &\leq
 \int_0^M \frac{h(\tau)}{2}\omega\Bigg(\sum_{k=1}^{N_\tau}\frac{y_k^\tau-x_k^\tau}{N_\tau}\Bigg)\,d\tau
  =\frac{1}{2}\int_0^M h(\tau)\omega\bigg(\frac{2g(\tau)}{h(\tau)}\bigg)\,d\tau \\
&\leq \int_0^M h(\tau)\omega\bigg(\frac{g(\tau)}{h(\tau)}\bigg)\,d\tau,
\end{align*}
the last inequality following from the fact that $\omega$ is concave with $\omega(0)=0$.
Dividing and multiplying this last integral by $|\gamma|_{TV}$ we can, by \eqref{eq:epigraph-integrals} and another application of Jensen's inequality, estimate further
\begin{align*}
\int_{\R}\gamma(x)\, d_x\beta(x) \leq |\gamma|_{TV}\omega\bigg(\int_0^M \frac{g(\tau)}{|\gamma|_{TV}}\,d\tau\bigg)
= |\gamma|_{TV}\omega\biggl(\frac{\|\gamma\|_{L^1}}{|\gamma|_{TV}}\biggr).
\end{align*}
Together, these calculations yield \eqref{eq: controlOnIntegralWithBVandOnesidedRegularity}.
\end{proof}

\subsection{Lower box envelope approximation}\label{sec: lowerBoxEnvelope}
This section is devoted to proving the following decomposition result.
\begin{lemma}[A monotone decomposition]\label{lem: monotoneDecomposition}
    Let $p\in[1,\infty)$ and let $\gamma\in V^p(\R)$ be non-negative and have support in an interval of length $L>0$. Let $r_0>0$ be a fixed parameter. Then there is a decomposition $\gamma=\sum_{k=0}^\infty\gamma_k$ such that
\begin{enumerate}[label=(\roman*)]
        \item the series converges at all continuity points of $\gamma$,
        \item each partial sum is of fixed $p$-variation, $\sup_{n}|\sum_{k=0}^n\gamma_k|_{V^p}\leq |\gamma|_{V^p}$,
        \item for each $k$ we have $\gamma_k\geq 0$ and the $L^1$-bound
\begin{subequations}\label{eq: estimatesOnLayers}
\begin{align}\label{eq: L1EstimatesOnLayers}
     \|\gamma_0\|_{L^1}\leq&\, \|\gamma\|_{L^1}, &  \|\gamma_k\|_{L^1}\leq&\, \frac{r_0^{\frac{1}{p}}L^{1-\frac{1}{p}}}{2^{\frac{k-2}{p}}}|\gamma|_{V^p},\quad k\in\N,
\end{align}
and the total variation bound
\begin{equation}\label{eq: TVEstimatesOnLayers}
    |\gamma_k|_{TV} \leq \bigg(\frac{2^{k}L}{r_0}\bigg)^{1-\frac{1}{p}}|\gamma|_{V^p},\quad k\in\N\cup\{0\}.
\end{equation}
\end{subequations}
\end{enumerate}
\end{lemma}

\begin{remark}
    The estimates \eqref{eq: L1EstimatesOnLayers} and \eqref{eq: TVEstimatesOnLayers} are --- up to some multiplicative constant --- the bounds one would obtain from a Littlewood--Paley decomposition of $\gamma$. The novelty in the previous result is the monotonicity: Each component $\gamma_k$ is non-negative, whereas the Littlewood--Paley blocks are of mixed sign. This detail is vital as we aim to apply Proposition \ref{prop: controlOnIntegralWithBVandOnesidedRegularity} to each such $\gamma_k$.
\end{remark}

The construction of our monotone decomposition relies on the following novel approximation procedure.

\begin{definition}[Lower box envelope]\label{def: definitionOfLowerBoxEnvelope}
   For $\gamma\colon \R\to [0,\infty)$ we define its lower box envelope of radius $r>0$ by
\begin{align}\label{eq: definitionOfLowerBoxEnvelope}
\overline{\gamma}_r(x)\coloneqq
\adjustlimits\sup_{z\in B_r(x)}\inf_{y\in B_r(z)}\gamma(y)
= \adjustlimits\sup_{|x_1|< r}\inf_{|x_2|<r}\gamma(x+x_1+x_2)
\end{align}
where $B_r(x) \coloneqq (x-r,x+r)$.
\end{definition}
The non-negativity assumption on $\gamma$ is merely for technical simplicity since the previous definition can then be interpreted as follows:
\begin{remark}
    The lower box envelope enjoys the following characterization. Let $X_{r}(\gamma)$ denote the set of box-functions of radius $r>0$ below $\gamma$, meaning $\Pi\in X_r(\gamma)$ if and only if there is a height $h\in[0,\infty)$ and centre $z\in\R$ such that
    \begin{equation*}
 \Pi(x)=\begin{cases}
            h &  \text{for }x\in B_r(z)\\
            0 & \text{for } x\notin B_r(z)
        \end{cases} \quad
        \text{and}\quad\Pi(x)\leq \gamma(x)\text{ for all }x\in\R.
    \end{equation*}
Note that $X_r(\gamma)$ is non-empty since it contains $0$. Note also that $\Pi\leq \gamma$ on $B_r(z)$ exactly when $h\leq \inf_{y\in B_r(z)}\gamma(y)$; this, together with the fact that $\Pi(x)=0$ whenever $z\notin B_r(x)$, implies that \eqref{eq: definitionOfLowerBoxEnvelope} is equivalent to
\begin{align}\label{eq: theGeometricCharacterizationOfTheLowerBoxEnvelope}
    \br(x)= \sup_{\Pi\in X_r(\gamma)}\Pi(x).
\end{align}
\end{remark}
We supplement this definition with a lemma listing some useful properties of this approximation.  We stress that throughout the section we use the set-theoretic definition of support,
\begin{align*}
    \supp(f)\coloneqq \{x: f(x)\neq 0\},
\end{align*}
that is, $\supp(f)$ need not be closed.

\begin{lemma}[Properties of the lower box envelope]\label{lem:lower-box-envelope}
    Let  $\gamma\colon\R\to [0,\infty)$ be lower semi-continuous and let $r>0$. Then
    \begin{enumerate}[label=\textit{(\roman*)}]
        \item\label{lem:lower-box-envelope-monotone}
        $0\,\leq \,\br\,\leq \,\overline{\gamma}_{r'}\,\leq \,\gamma$ whenever $0<r'<r$, and $\lim_{r\downarrow 0}\br=\gamma$ pointwise,
        \item\label{lem:lower-box-envelope-lsc}
        $\br$ is lower semi-continuous,
        \item\label{lem:lower-box-envelope-open-supp}
        $\supp(\gamma-\br)$ is open and $\br$ is locally constant on this set,
        \item\label{lem:lower-box-envelope-size-supp}
        the connected components of $\supp(\gamma-\br)$ are of width at most $2r$, and
        \item\label{lem:lower-box-envelope-maxima} the local maxima of $\br$ are of width at least $2r$, in the sense that for every triplet $x_0<x_1<x_2$ we have the implication
    \begin{align*}
        \br(x_0)<\br(x_1)>\br(x_2)\qquad\implies\qquad x_2-x_0 > 2r.
    \end{align*}
    \end{enumerate}
\end{lemma}
\begin{proof}
We first prove the monotonicity property \ref{lem:lower-box-envelope-monotone}: From the characterization \eqref{eq: theGeometricCharacterizationOfTheLowerBoxEnvelope} it is clear that $0\leq \br\leq \gamma$ since each $\Pi$ satisfies $0\leq \Pi\leq \gamma$. Moreover, any box $\Pi\in X_r(\gamma)$, is the envelope of thinner boxes
\begin{equation*}
    \Pi(x)= \sup_{\pi\in X_{r'}(\Pi)}\pi(x)
\end{equation*}
for any $r'<r$, and since $\Pi\leq \gamma$ we also have $X_{r'}(\Pi)\subset X_{r'}(\gamma)$. It follows that $\overline{\gamma}_r\leq\overline{\gamma}_{r'}$.
The pointwise convergence as $r\downarrow 0$ is ensured by the lower semi-continuity of $\gamma$.

Next, to prove \ref{lem:lower-box-envelope-lsc} we note that each box $\Pi\in X_r(\gamma)$ is lower semi-continuous, and as this property is preserved when taking the supremum over such a family, \ref{lem:lower-box-envelope-lsc} is a consequence of the characterization \eqref{eq: theGeometricCharacterizationOfTheLowerBoxEnvelope}.

To prove \ref{lem:lower-box-envelope-open-supp}, suppose $x$ is such that $\gamma(x)> \br(x)$.
As $\gamma$ is lower semi-continuous there is some $\delta\in(0,r]$ such that $\inf_{y\in B_{\delta}(x)}\gamma(y)> \br(x)$. We aim to demonstrate that
\begin{align}\label{eq: elementsOfSuppAreNorSmallerThanTheirNeighbours}
    \br(y)\leq \br(x)\text{ for }y\in B_\delta(x),
\end{align}
as property \ref{lem:lower-box-envelope-open-supp} would then follow; the set $\supp(\gamma-\br)$ would immediately be open, and if $x_1,x_2\in \supp(\gamma-\br)$ belonged to the same connected component we would get $\br(x_1)\leq \br(x_2)\leq \br(x_1)$, so $\br$ would be constant on that component. We prove \eqref{eq: elementsOfSuppAreNorSmallerThanTheirNeighbours} by contradiction: Assume there is an $\tilde x\in B_\delta(x)$ and a box function $\Pi\in X_r(\gamma)$ such that $h\coloneqq\Pi(\tilde x)>\br(x)$. By possibly lowering the height $h$, we may also assume that $h<\inf_{y\in B_{\delta}(x)}\gamma(y)$. Clearly, $x\notin \supp(\Pi)$. However, the shifted box function
\begin{align*}
    \tilde{\Pi}(y)\coloneq\Pi\big(y + \tilde{x}-x\big),
\end{align*}
\textit{does} include $x$ in its support. Furthermore, $\tilde{\Pi}\in X_r(\gamma)$ because $\supp(\tilde\Pi)\setminus \supp (\Pi)\subset B_\delta(x)$ and $\tilde\Pi(y)\leq h<\gamma(y)$ for all $y\in B_{\delta}(x)$. Thus, we get the contradiction $\br(x)\geq \tilde{\Pi}(x)=h>\br(x)$, and so \eqref{eq: elementsOfSuppAreNorSmallerThanTheirNeighbours} follows, thus proving \ref{lem:lower-box-envelope-open-supp}.

For \ref{lem:lower-box-envelope-size-supp}, pick a connected component $I\subset \supp(\gamma-\br)$ and fix a point $x\in I$. By the definition \eqref{eq: definitionOfLowerBoxEnvelope} of $\br(x)$ and the lower semi-continuity of $\gamma$ we have
\begin{align*}
    \br(x)\geq \inf_{y\in B_r(x)}\gamma(x) = \gamma(y_0),
\end{align*}
for some $y_0\in B_r(x)$. As $\gamma>\br\equiv\br(x)$ on $I$, it follows that $y_0\notin I$, and so $\dist(x,I^c)\leq |x-y_0|\leq r$; as the same can be said for any point in $I$ we conclude that $|I|\leq 2r$.

As for \ref{lem:lower-box-envelope-maxima}, we can, by the definition of $\br(x_1)$, pick $z\in B_r(x_1)$ such that
\begin{align*}
    \br(x_1)\geq \inf_{y\in B_r(z)}\gamma(y)>\max(\br(x_0), \br(x_2)),
\end{align*}
meaning that $z\notin B_r(x_0)\cup B_r(x_2)$, and so $x_2-x_0 = x_2-z + z-x_0> 2r$.
\end{proof}

\begin{proposition}\label{prop: TVBoundandL1ConvergenceOfWidthApproximation}
        Let $\gamma\colon \R\to [0,\infty)$ be lower semi-continuous, have finite $p$-variation with $p\in[1,\infty)$, and be supported in some interval $A=[a,a+L]$. We then have the estimates
    \begin{align}\label{eq: qVariationBoundOnFMApproximation}
      |\br|_{V^q(A)}\leq&\, |\gamma|_{V^p(A)}\bigg(\frac{L}{r}\bigg)^{\frac{1}{q}-\frac{1}{p}},\\\label{eq: L1ControllOnFMApproximation}
      \|\gamma-\br\|_{L^q(A)}\leq&\, |\gamma|_{V^p(A)}(2r)^{\frac{1}{p}}L^{\frac{1}{q}-\frac{1}{p}},
    \end{align}
 for any $q\in[1,p]$ and $r>0$.
\end{proposition}
\begin{proof}
By lower semi-continuity, we have $\gamma(a)=0=\gamma(a+L)$, and thus the same is true for $\br$. Since $\br$ is non-negative, we may restrict our attention to finite sequences $a=x_0<x_1<\dots < x_{2n}=a+L$ such that
\begin{align}\label{eq: theOscilatingSequence}
    \br(x_0)<\br(x_1)>\br(x_2)<\dots > \br(x_{2n}),
\end{align}
when seeking to  estimate the $q$-variation of $\br$. Furthermore, for the even-numbered points $x_{2k}$ we shall assume that $\br(x_{2k})= \gamma(x_{2k})$. If this was not the case for some $x_{2k}$, then we could pick the (open) connected component $I\subset \supp (\gamma-\br)$ containing $x_{2k}$ (on which $\br$ is constant) and replace $x_{2k}$ by the endpoint $\tilde{x}_{2k}=\inf I$. Since $\tilde{x}_{2k}\notin \supp(\gamma-\br)$ we get $\gamma(\tilde{x}_{2k})=\br(\tilde{x}_{2k})$, and by lower semi-continuity we also have $\br(\tilde{x}_{2k})\leq \br(x_{2k})$; thus, we are only enlarging the oscillations in \eqref{eq: theOscilatingSequence}. Together with the trivial fact that $\br(x_{2k+1})\leq\gamma(x_{2k+1})$ we conclude that
\begin{align}\label{eq: oscillatingSequencesOfBWAreDominatedByThatOfB}
    \Delta_k\br\coloneqq |\br(x_{k})-\br(x_{{k-1}})|\leq |\gamma(x_{k})-\gamma(x_{{k-1}})|\eqqcolon \Delta_k\gamma
\end{align}
for all $k$. The number of summands can also be bounded: By Lemma \ref{lem:lower-box-envelope} \ref{lem:lower-box-envelope-size-supp} we have
\begin{align}\label{eq: boundOnNumberOfSummands}
    L  =\sum_{k=1}^{n}x_{2k}-x_{2k-2}\geq 2rn,\qquad \implies \qquad 2n\leq L/r.
\end{align}
 By \eqref{eq: oscillatingSequencesOfBWAreDominatedByThatOfB}, \eqref{eq: boundOnNumberOfSummands} and Hölder's inequality, we get
\begin{align*}
    \|\Delta_{\subscriptcdot} \br\|_{\ell^q}\leq \|\Delta_{\subscriptcdot} \gamma\|_{\ell^q}\leq  \|\Delta_{\subscriptcdot} \gamma\|_{\ell^p}(2n)^{\frac{1}{q}-\frac{1}{p}}\leq |\gamma|_{V^p}\bigg(\frac{L}{r}\bigg)^{\frac{1}{q}-\frac{1}{p}},
\end{align*}
and so the $q$-variation bound follows.

Next, we prove the $L^q$-difference bound \eqref{eq: L1ControllOnFMApproximation}: Pick $n$ connected components $(x_k,y_k)$ (for $k=1,\dots,n$) from $\supp(\gamma-\br)$ and assume, without loss of generality, that $y_k < x_{k+1}$ for $k=1,\dots,n-1$. Recall from Lemma \ref{lem:lower-box-envelope} \ref{lem:lower-box-envelope-open-supp} that $\br$ is constant on each component $(x_k,y_k)$. We estimate
\begin{align}\label{eq: aBeginningOfTheL1Control}
    \sum_{k=1}^n\int_{x_k}^{y_k}|\gamma(z)-\br(z)|^q\,dz\leq \sum_{k=1}^n(y_k-x_k)\sup_{z\in(x_k,y_k)}\big(\gamma(z)-\br(z)\big)^q.
\end{align}
By lower semi-continuity of $\br$ we have that $\br(x_k)$ is no larger than the constant $\overline{\gamma}_{r}(z)$, for $z\in (x_k,y_k)$, and since $x_k\notin\supp(\gamma-\br)$ we also have $\gamma(x_k)=\br(x_k)$. For $\epsilon>0$ we can thus pick $z_k\in(x_k,y_k)$ such that
\begin{align}\label{eq: simpleBoundOnHeight}
    \sup_{z\in(x_k,y_k)}\frac{\big(\gamma(z)-\br(z)\big)^q}{1+\epsilon} \leq \big(\gamma(z_k)-\gamma(x_k)\big)^q.
\end{align}
Setting $\Delta_k x\coloneqq y_k-x_k$ and $\Delta_k \gamma\coloneqq \gamma(z_k)-\gamma(x_k)$ we get, by \eqref{eq: aBeginningOfTheL1Control}, \eqref{eq: simpleBoundOnHeight}, Hölder's inequality, and Lemma \ref{lem:lower-box-envelope} \ref{lem:lower-box-envelope-size-supp}, that
\begin{align*}
 \Bigg(\sum_{k=1}^n\int_{x_k}^{y_k} \frac{\big(\gamma(z)-\br(z)\big)^q }{1+\epsilon}\,dz\Bigg)^{\frac{1}{q}} &\leq \|(\Delta_{\subscriptcdot} x)^{\frac{1}{q}}\Delta_{\subscriptcdot} \gamma\|_{\ell^{q}}\\
    &\leq \|(\Delta_{\subscriptcdot} x)^{\frac{1}{q}}\|_{\ell^{\frac{pq}{p-q}}}\|\Delta_{\subscriptcdot}\gamma\|_{\ell^p} \\
      &\leq \|\Delta_{\subscriptcdot} x\|_{\ell^{\infty}}^{\frac{1}{p}}\|\Delta_{\subscriptcdot} x\|_{\ell^{1}}^{\frac{1}{q}-\frac{1}{p}}|\gamma|_{V^p} \\
      &\leq (2r)^{\frac{1}{p}}L^{\frac{1}{q}-\frac{1}{p}}|\gamma|_{V^p}.
\end{align*}
The right-hand side is independent of both $\varepsilon>0$ and our finite selection of connected components from $\supp(\gamma-\overline{\gamma})$: By a limit argument (and Fatou's lemma) we thus get the result.
\end{proof}

We can now prove the desired decomposition result:
\begin{proof}[Proof of Lemma \ref{lem: monotoneDecomposition}]
Because $\gamma$ is of finite $p$-variation with $p<\infty$, it has at most countably many discontinuities; thus, by modifying $\gamma$ on this null set, we can assume it is lower semi-continuous. Define now
    \begin{align*}
        \gamma_0\coloneqq \overline{\gamma}_{r_0}\qquad\text{and}\qquad \gamma_k \coloneqq \overline{\gamma}_{r_0/2^k} - \overline{\gamma}_{r_0/2^{k-1}} \quad(k\in\N),
    \end{align*}
with $\br$ as in Definition \ref{def: definitionOfLowerBoxEnvelope}. Each $\gamma_{k}$ is non-negative by Lemma \ref{lem:lower-box-envelope} \ref{lem:lower-box-envelope-monotone}, and
\begin{align}\label{eq: pointwiseConvergenceOfLayers}
    \sum_{k=0}^n \gamma_k(x) = \overline{\gamma}_{r_0/2^{n}}(x) \to \gamma(x)\qquad \text{as }n\to\infty.
\end{align}
Moreover, by \eqref{eq: qVariationBoundOnFMApproximation} from Proposition \ref{prop: TVBoundandL1ConvergenceOfWidthApproximation}, we have
 \begin{align*}
\sup_{n\in\N}\,\Biggl|\sum_{k=0}^n\gamma_k\Biggr|_{V^p}\!\! = \sup_{n\in\N}\,\abs{\overline{\gamma}_{r_0/2^{n}}}_{V^p}\leq |\gamma|_{V^p}.
\end{align*}
To prove the $L^1$-bound \eqref{eq: L1EstimatesOnLayers}, we again use Lemma \ref{lem:lower-box-envelope} \ref{lem:lower-box-envelope-monotone} to see that
\begin{align*}
    0\leq \gamma_0 \leq \gamma \qquad \text{and}\qquad 0\leq \gamma_{k}\leq \gamma - \overline{\gamma}_{r_0/2^{k-1}} \quad (k\in\N).
\end{align*}
Then \eqref{eq: L1EstimatesOnLayers} follows from the $L^1$-bound \eqref{eq: L1ControllOnFMApproximation}.

It remains to prove the $TV$-bound \eqref{eq: TVEstimatesOnLayers}. For this, we first claim that $ |\gamma_k|_{TV}\leq  |\overline{\gamma}_{r_0/2^{k}}|_{TV}$ for any $k=0,1,\dots$ Indeed, this holds with equality when $k=0$, and for (a fixed) $k\geq 1$ we argue as follows:
By Lemma \ref{lem:lower-box-envelope} \ref{lem:lower-box-envelope-open-supp}, there is a countable family of disjoint open intervals $I_n$ on which $\overline{\gamma}_{r_0/2^{k-1}}$ is locally constant, and such that $\overline{\gamma}_{r_0/2^{k-1}}=\overline{\gamma}_{r_0/2^{k}}=\gamma$ on $\R\setminus \bigcup_n I_n$. Thus,
\begin{align*}
    |\gamma_{k}|_{TV} = \sum_{n} |\overline{\gamma}_{r_0/2^{k}}-\overline{\gamma}_{r_0/2^{k-1}}|_{TV(I_n)}= \sum_{n} |\overline{\gamma}_{r_0/2^{k}}|_{TV(I_n)}\leq |\overline{\gamma}_{r_0/2^{k}}|_{TV}.
\end{align*}
We then get \eqref{eq: TVEstimatesOnLayers} by applying the $TV=V^1$-bound \eqref{eq: qVariationBoundOnFMApproximation} to $|\overline{\gamma}_{r_0/2^{k}}|_{TV}$.
\end{proof}

\subsection{Proof of Theorem~\ref{thm: controlOnIntegralWithPVariationandOnesidedRegularity}}\label{sec:controlling-transport-term}

We can now prove the main result of this section.

\begin{proof}[Proof of Theorem~\ref{thm: controlOnIntegralWithPVariationandOnesidedRegularity}]
Decompose $\gamma = \sum_{k=0}^\infty\gamma_k$ as in Lemma \ref{lem: monotoneDecomposition}. Since this series converges pointwise and with uniformly bounded $p$-variance, Theorem~\ref{thm:young-integral-properties}~\ref{thm:young-integral-properties_pw-convergence} implies that
\[
\int_\R \gamma\,d_x\beta = \sum_{k=0}^\infty \int_\R \gamma_k\,d_x\beta.
\]
We estimate each summand using Proposition \ref{prop: controlOnIntegralWithBVandOnesidedRegularity}. For $k=0$ we get, for any $r_0>0$,
\begin{align*}
&\int_\R \gamma_0\,d_x\beta \leq |\gamma_0|_{TV}\omega\biggl(\frac{\|\gamma_0\|_{L^1}}{|\gamma_0|_{TV}}\biggr)
\leq |\gamma|_{V^p}\frac{L^{1-\frac1p}}{r_0^{1-\frac1p}} \omega\Biggl(\frac{\|\gamma\|_{L^1}r_0^{1-\frac{1}{p}}}{|\gamma|_{V^p}L^{1-\frac{1}{p}}}\Biggr),
\end{align*}
where $r_0>0$ is arbitrary, and where used \eqref{eq: estimatesOnLayers} and the fact that both $\omega$ and $r\mapsto r\omega(C/r)$ are increasing for any $C\geq0$. For $k\geq1$ we estimate similarly
\begin{align*}
\int_\R \gamma_k\,d_x\beta &\leq |\gamma_k|_{TV}\omega\biggl(\frac{\|\gamma_k\|_{L^1}}{|\gamma_k|_{TV}}\biggr) \\
&\leq |\gamma|_{V^p}\bigg(\frac{2^{k}L}{r_0}\bigg)^{1-\frac{1}{p}}\omega\Biggl(\frac{1}{|\gamma|_{V^p}}\bigg(\frac{r_0^{1-\frac{1}{p}}}{2^{k(1-\frac{1}{p})}L^{1-\frac{1}{p}}}\bigg)\bigg(\frac{|\gamma|_{V^p}r_0^{\frac1p}L^{1-\frac{1}{p}}}{2^{\frac{k-2}{p}}}\bigg)\Biggr) \\
&= |\gamma|_{V^p}\bigg(\frac{2^{k}L}{r_0}\bigg)^{1-\frac{1}{p}} \omega\bigg(\frac{r_0}{2^{k-\frac{2}{p}}}\bigg).
\end{align*}
Using Assumption \ref{ass:oshc}, let $C_p>0$ be such that $\omega(h)\leq C_p h^{1-\frac{1}{2p}}$ for all $h>0$. We obtain
\begin{align*}
\sum_{k=1}^\infty\int_\R \gamma_k\,d_x\beta
&\leq |\gamma|_{V^p}\frac{L^{1-\frac1p}}{r_0^{1-\frac1p}}\sum_{k=1}^\infty 2^{k(1-\frac{1}{p})} \omega\bigl(2^{\frac{2}{p}-k}r_0\bigr) \\
&\leq |\gamma|_{V^p}L^{1-\frac{1}{p}}r_0^{\frac{1}{2p}}C_p\bigg(\sum_{k=1}^\infty 2^{k(1-\frac{1}{p})+(1-\frac{1}{2p})(\frac{2}{p}-k)}\bigg) \\
&= |\gamma|_{V^p}L^{1-\frac{1}{p}}r_0^{\frac{1}{2p}}C_p\bigg(\frac{2^{\frac{1}{p}(2-\frac{1}{p})}}{2^{\frac{1}{2p}}-1}\bigg)\eqqcolon c_p|\gamma|_{V^p}L^{1-\frac{1}{p}}r_0^{\frac{1}{2p}}.
\end{align*}
Summarizing, we have
\begin{align*}
    \int_\R \gamma\,d_x\beta \leq |\gamma|_{V^p}\frac{L^{1-\frac1p}}{r_0^{1-\frac1p}}\omega\Biggl(\frac{\|\gamma\|_{L^1}r_0^{1-\frac{1}{p}}}{|\gamma|_{V^p}L^{1-\frac{1}{p}}}\Biggr)+c_p|\gamma|_{V^p}L^{1-\frac{1}{p}}r_0^{\frac{1}{2p}},
\end{align*}
which, by straight forwards algebra, coincides with the right-hand side of \eqref{eq: controlOnIntegralWithPVariationandOnesidedRegularity} for the choice
\begin{equation*}
    r_0 =\frac{\|\gamma\|_{L^1}^{2p}}{|\gamma|_{V^p}^{2p}L^{2p-2}}.\qedhere
\end{equation*}
\end{proof}

\section{Well-posedness in one space dimension}
This section is devoted to the well-posedness of the one-dimensional forwards problem \eqref{eq:transport1d}. In Section \ref{sec:existence-1d-forwards} we show existence of a solution by means of a standard approximation argument. We show that $t\mapsto u(t)$ is H\"older continuous into $L^1(\R)$ in Section~\ref{sec:holder-time-continuity}. Section \ref{sec:forwards-uniqueness} is devoted to uniqueness: We show that solutions are renormalizable in Section~\ref{sec:renormalization}; and show $L^1$-stability and uniqueness in Section~\ref{sec:uniqueness-and-stability}.

\subsection{Existence}\label{sec:existence-1d-forwards}
\begin{theorem}\label{thm:existence-of-solution}
    Let $b$ and $u_0$ satisfy Assumptions \ref{ass:b-conditions} and \ref{ass:uzero-conditions}, respectively. Let $X_t(x)=X_t(x,0)$ be the flow of the velocity field $b$ (as defined in Section \ref{sec: ODEsWithOsgoodVelocityFields}). Then, for all $t\geq 0$, the function
    \begin{align*}
        u(x,t) \coloneqq u_0\bigl(X_t^{-1}(x)\bigr)
    \end{align*}
    is well defined for a.e.~$x\in \R$, and $u$ is a weak solution of \eqref{eq:transport1d} in the sense of Definition \ref{def: definitionOfWeakSolution}. Moreover, if $u^\eps$ is the solution of the equation with mollified velocity $[b]^\eps$ and data $[u_0]^\eps$, then $u^\eps\to u$ pointwise as $\eps\to0$.
\end{theorem}

\begin{proof}
    By assumption \ref{ass:regularityOfInitialData} and Proposition \ref{prop:p-variation} \ref{prop:p-variation-limits}, we may assume that $u_0$ is right-continuous. In particular, the discontinuities of $u_0$ form a countable set $D$. Fix $t\geq 0$. By Theorem~\ref{thm:ode-well-posed}, the function $x\mapsto X_t(x)$ is non-decreasing and surjective, so $X_t^{-1}$ is ill-defined at an (at most) countable set of point that we denote $E_t$; on the remaining set $\R\setminus E_t$, $X_t^{-1}$ is well-defined, injective and continuous. To combine the bad set of $u_0$ with that of $X_t$ we define the countable set
    \begin{equation*}
        F_t\coloneqq X_t(D)\cup E_t,
    \end{equation*}
    and conclude that $u(\cdot,t)=u_0\circ X_t^{-1}$ is well-defined and continuous on $\R\setminus F_t$.

 The fact that $u$ solves the equation in the sense of Definition \ref{def: definitionOfWeakSolution} follows by an approximation argument, as follows. Recalling the notation in Section \ref{sec:notation}, we let $[b]^\eps$, $[u_0]^\eps$ be spatial mollifications of $b$, $u_0$, and $X_\eps$ the flow of $[b]^\eps$. As $b$ is bounded, the mapping $x\mapsto [b]^\eps(x,t)$ is uniformly (in $t$) Lipschitz continuous, and therefore $x\mapsto X_{\eps,t}(x)$ is bi-Lipschitz continuous with both $\partial_x X_{\eps,t}$ and $\partial_x (X_{\eps,t})^{-1}$ bounded by $e^{Ct/\varepsilon}$ for an appropriate constant $C$. The unique weak and locally bounded solution of $\partial_t u^\eps + [b]^\eps \partial_xu^\eps = 0$ with $u(\cdot,0)=[u_0]^\eps$ is then given by
 \begin{equation*}
     u^\eps(x,t)\coloneqq [u_0]^\eps\bigl((X_{\eps,t})^{-1}(x)\bigr).
 \end{equation*}
 That is, the integral equation
\begin{equation}\label{eq: epsilonSolutionsSolvesTheEquationWeakly}
    \int_0^{\infty}\int_{\R} u^\eps \partial_t\phi\, dx\, dt + \int_0^{\infty}\int_{\R} u^\eps\,d_x ([b]^\eps \phi) \, dt + \int_{\R} [u_0]^\eps\phi\, dx = 0,
\end{equation}
 holds for every $\varphi\in C_c^\infty(\R\times[0,\infty))$, where the Riemann--Stieltjes integral in the middle can be interpreted as a standard Lebesgue integral since $d_x([b]^\varepsilon \varphi)=\partial_x([b]^\varepsilon \varphi)\, dx$.
 It remains to prove that $u_0,b,u$ can replace $[u_0]^\eps,[b]^\eps,u^\eps$ in  \eqref{eq: epsilonSolutionsSolvesTheEquationWeakly}. We first claim the following:
 \begin{enumerate}[label=\textit{(\roman*)}]
     \item $[u_0]^\eps\to u_0$ in $L^1_{\loc}(\R)$,
     \item $[b]^\eps\to b$ locally uniformly in $x$ for a.e.~$t>0$, and
     \item $u^\eps\to u$ pointwise whenever $x\in \R\setminus F_t$ and $t>0$.
 \end{enumerate}
The first two limits are obvious (since $u_0\in L^1_\loc(\R)$, and since $b$ is continuous in $x$), and so we focus on proving the third. Fix $t>0$ and $y\in \R\setminus F_t$ and set $y_0\coloneqq (X_t)^{-1}(y)$. Since $F_t$ is countable, we can for every $\delta>0$ pick $x_\delta,z_\delta\in \R\setminus F_t$ such that $x_\delta<y<z_\delta$ and $|z_\delta-x_\delta|<\delta$. By monotonicity and injectivity, we also have $x_{0,\delta}\coloneqq (X_t)^{-1}(x_\delta)<y_0 < (X_t)^{-1}(z_\delta)\eqqcolon z_{0,\delta}$. We know that $X_{\eps,t}\to X_t$ pointwise as $\eps\to0$ (see~Theorem~\ref{thm:ode-well-posed}), and so
 \begin{align*}
    \lim_{\varepsilon \downarrow 0}X_{\eps,t}(x_{0,\delta}) = x_\delta\qquad &\implies \qquad \liminf_{\varepsilon\downarrow 0} (X_{\eps,t})^{-1}(y)\geq x_{0,\delta},\\
    \lim_{\varepsilon \downarrow 0}X_{\eps,t}(z_{0,\delta})=z_\delta\qquad &\implies \qquad \limsup_{\varepsilon\downarrow 0} (X_{\eps,t})^{-1}(y)\leq z_{0,\delta},
 \end{align*}
 where the implications follow from monotonicity of $X_{\eps,t}$. Since $X_t^{-1}$ is continuous at $y$ it follows that $x_{0,\delta},z_{0,\delta}\to y_0$ as $\delta\downarrow 0$, and so $\lim_{\varepsilon\downarrow 0}(X_{\eps,t})^{-1}(y)= X^{-1}_t(y)$. Finally, as $u_0$ is continuous at $X^{-1}_t(y)\in \R\setminus D$, we also get $\lim_{\varepsilon\downarrow 0}[u_0]^\eps(X_{\eps,t}^{-1}(y))= u_0(X^{-1}_t(y))$.

With the above established limits, it remains to prove uniform $p$-variation estimates on $u^\eps$ and $[b]^\eps\varphi$ (where we consider $\varphi$ fixed). For this, let $R,T>0$ be constants such that $\supp \varphi \subset [-R,R]\times (-\infty,T]$. By assumption \ref{ass:regularityOfInitialData}, there is some $p\in[1,\infty)$ such that
\begin{align*}
    |u_0|_{V^p([-\tilde{R},\tilde{R}])}<\infty,
\end{align*}
where $\tilde{R}\coloneqq R + T\|b\|_{L^\infty} + 1$. By monotonicity and finite speed of propagation of the flow $X_{\eps,t}$, it is easy to see that we for every $\varepsilon\in(0,1)$ have
\begin{align}\label{eq: uniformEstimatesForUEpsilon}
    \sup_{0<t<T}|u^\eps(t)|_{V^p([-R,R])}\leq |u_0|_{V^p([-\tilde{R},\tilde{R}])}.
\end{align}
For $[b]^\eps\varphi$, we write $\varphi_+$ and $\varphi_-$ for the positive and negative part of $\varphi$ respectively, and note that both  $[b]^\eps\varphi_+$ and $[b]^\eps\varphi_-$ admit $(t,h)\mapsto \|b\|_{L^\infty}\|\partial_x\varphi\|_{L^\infty} h + \|\varphi\|_{L^\infty}\lambda(t)\omega_b(h)$ as a one-sided continuity modulus; applying then Corollary \ref{cor: finitePVariationOfTheVelocityB} we get
\begin{align}\label{eq: uniformEstimatesForBEpsilon}
     \bigl|[b]^\eps(t)\varphi(t)\bigr|_{V^q([-R,R])}\leq C(1+\lambda(t)),
\end{align}
for a.e.~$t\in[0,\infty)$, where we have fixed $q\in[1,\frac{p}{p-1})$, and where $C$ is independent of $\varepsilon$ and $t$.
Then, by the above three limits, the uniform estimates \eqref{eq: uniformEstimatesForUEpsilon} and \eqref{eq: uniformEstimatesForBEpsilon}, and the equation \eqref{eq: epsilonSolutionsSolvesTheEquationWeakly}, we can send $\varepsilon\downarrow 0$ and conclude by Theorem~\ref{thm:young-integral-properties}~\ref{thm:young-integral-properties_size-control} and \ref{thm:young-integral-properties_pw-convergence} and by dominated convergence that $u$ is a solution of \eqref{eq:transport1d} in the sense of Definition \ref{def: definitionOfWeakSolution}.
\end{proof}

\subsection{Time continuity}\label{sec:holder-time-continuity}

\begin{proposition}\label{prop:time-continuity}
    Let $u$ be a solution of \eqref{eq:transport1d} in the sense of Definition \ref{def: definitionOfWeakSolution}. Then $u\in C\big([0,\infty), L^1_{\loc}(\R)\big)$.
\end{proposition}
\begin{proof}
Fix $R,T>0$; it suffices to prove that $u\in C\big([0,T], L^1([-R,R])\big)$. If we in \eqref{eq: theWeakRiemannStieltjesFormOfTheEquation} relabel $x\mapsto y$ and insert the test function $(y,t)\mapsto \rho_\varepsilon(x-y)\sigma(t)$, where $\rho_\varepsilon$ is a standard mollifier, $\sigma\in C_c^\infty([0,T))$, and $x\in[-R,R]$ is a fixed parameter, then we get
    \begin{equation*}
           \int_{0}^T [u]^\varepsilon (x,t)\partial_t\sigma(t)\, dt + [u_0]^\varepsilon(x)\sigma(0)= \int_0^T\int_\R\rho_\varepsilon(x-y)b(y,t) \, d_yu(y,t)\sigma(t)\, dt,
    \end{equation*}
where $[u]^\varepsilon$,$[u_0]^\varepsilon$ denote the spatial mollifications of $u,u_0$, and where we integrated by parts (Theorem \ref{thm:young-integral-properties} \ref{thm:young-integral-properties-integration-by-parts}) on the right-hand side. Note that we can safely assume $\supp\rho_{\varepsilon}(x-\cdot)\subseteq A\coloneqq [-(R+1),R+1]$ for all $\varepsilon\in(0,1)$ and $x\in[-R,R]$. Applying Theorem~\ref{thm:young-integral-properties}~\ref{thm:young-integral-properties_size-control} (with $\theta=0$) and Proposition~\ref{prop:p-variation}~\ref{prop:p-variation-product-rule}, we obtain the bound
\begin{align*}
    &\bigg|\int_\R \rho_\varepsilon(x-y)b(y,t)\, d_yu(y,t)\bigg|= \bigg|\int_A \rho_\varepsilon(x-y)b(y,t)\, d_yu(y,t)\bigg|\\ &\qquad\leq \|\rho_\varepsilon\|_{L^\infty}\|b\|_{L^\infty}\|u(\cdot,t)\|_{L^\infty(A)} \\
    &\qquad\quad + C_{p,q} |u(\cdot,t)|_{V^p(A)} \Big(|\rho_\varepsilon(x-\cdot)|_{V^{q}(A)}\|b\|_{L^\infty}  +\|\rho_\varepsilon\|_{L^\infty}|b(\cdot,t)|_{V^q(A)}\Big)\\
    &\qquad\lesssim_{R,T,p,q} \frac{1+\lambda(t)}{\varepsilon}
\end{align*}
where $p$ is as in \eqref{eq: regularityAssumptionOnSolution} and $q\in[1,\frac{p}{p-1})$, and where we used Corollary \ref{cor: finitePVariationOfTheVelocityB}. Inserting this above yields
\begin{align*}
    \bigg|\int_{0}^\infty [u]^\varepsilon (x,t)\sigma'(t)\, dt + [u_0]^\varepsilon(x)\sigma(0)\bigg|\lesssim_{R,T,p,q} \int_0^T\frac{(1+\lambda(t))}{\varepsilon}|\sigma(t)|\,dt.
\end{align*}
By standard arguments it follows  that $t\mapsto [u]^\varepsilon (x,t)$ is absolutely continuous on $[0,T]$, with a weak derivative satisfying $|\partial_t [u]^\varepsilon (x,t)|\leq C(1+\lambda(t))/\varepsilon$ for a $C$ independent of $x\in[-R,R]$, $t\in[0,T]$, and $\varepsilon\in(0,1)$, and where $[u]^{\varepsilon}(x,0)=[u_0]^\varepsilon(x)$. Thus, for $0\leq s\leq t\leq T$ we get
\begin{align*}
    \bigl\|[u]^\varepsilon(t)-[u]^\varepsilon(s)\bigr\|_{L^1([-R,R])}\leq C\frac{2R}{\varepsilon}\int_{s}^{t}1+\lambda(\tau)\, d\tau.
\end{align*}
Furthermore, for a.e.~$t\in[0,T]$ we have
\begin{align*}
    \|u(t)-[u]^\varepsilon(t)\|_{L^1([-R,R])}\leq &\,\int_{|y|\leq \varepsilon}\frac{\rho(y/\varepsilon)}{\varepsilon}\bigg(\int_{-R}^R|u(x,t)-u(x-y,t)|\, dx\bigg) dy\\
   \leq &\, \|\rho\|_{L^1(\R)}(2R)^{1-\frac{1}{p}}\varepsilon^{\frac{1}{p}}|u(t)|_{V^p([-R,R])}
\end{align*}
where we for the inner integral used Hölder's inequality and Lemma \ref{lem:l1-translation-error-estimate}. By \eqref{eq: regularityAssumptionOnSolution} and the triangle inequality, these estimates yield
\begin{align*}
    \|u(t)-u(s)\|_{L^1[-R,R]}
    \leq \tilde C\Bigg( \varepsilon^{\frac{1}{p}} +\eps^{-1} \int_s^t 1+\lambda(\tau)\, d\tau\Bigg),
\end{align*}
for a.e.~$0\leq s\leq t\leq T$ and all $\varepsilon\in(0,1)$, where  $\tilde{C}$ is some large constant independent of $s,t,\varepsilon$. The result now follows by setting $\varepsilon = \big(\int_s^t1+\lambda(\tau)\,d\tau\big)^{\frac{p}{p+1}}$ and modifying $t\mapsto u(t)$ appropriately on a null set.
\end{proof}

\subsection{Uniqueness}\label{sec:forwards-uniqueness}
The uniqueness argument is much more demanding than the existence argument. It is therefore split over several subsections.

\subsubsection{Renormalization}\label{sec:renormalization}
\begin{theorem}[Renormalization of weak solutions]\label{thm:renormalizability}
Let $u$ be a weak solution of \eqref{eq:transport1d}. Then $\eta\circ u$ is also a weak solution, for any Lipschitz function $\eta\from\R\to\R$.
\end{theorem}

\begin{proof}
The $p$-variation estimate \eqref{eq: regularityAssumptionOnSolution} on $\eta\circ u$ follows from boundedness of $u$ and Lipschitz continuity of $\eta$.

Assume first that $\eta\in C^2(\R)$ with bounded derivatives. Recalling the notation in Section \ref{sec:notation}, we let $[u]^{\eps,\delta}$ denote space-time mollification of $u$ and $[u]^\eps$ spatial mollification.
By inserting $[u]^{\eps,\delta}$ into the weak formulation of \eqref{eq:transport1d}, we find that $[u]^{\eps,\delta}$ satisfies
\[
\partial_t [u]^{\eps,\delta} + b\partial_x [u]^{\eps,\delta} = r_{\eps,\delta} \qquad\text{for } x\in\R, t>\delta
\]
in the classical sense, with
\[
r_{\eps,\delta}(x,t) \coloneqq \int_0^\infty \int_\R \rho_\eps(x-y)\rho_\delta(t-s)\bigl(b(x,t) - b(y,s)\bigr)\,d_yu(y,s)\,ds,
\]
where we have integrated by parts (Theorem \ref{thm:young-integral-properties} \ref{thm:young-integral-properties-integration-by-parts}). In particular
\begin{equation}\label{eq:mollified-transport-equation}
\partial_t \eta([u]^{\eps,\delta}) + b\partial_x \eta([u]^{\eps,\delta}) = r_{\eps,\delta}\eta'([u]^{\eps,\delta}).
\end{equation}
Let $\phi\in C_c^\infty(\R\times(0,\infty))$, and let $\delta>0$ be small enough that $\supp\phi\subset\R\times[\delta,T-\delta]$ for some $T>0$. (The strong temporal continuity from Proposition~\ref{prop:time-continuity} implies that we only need to consider test functions supported away from $t=0$.) Multiplying \eqref{eq:mollified-transport-equation} by $\phi$ and integrating yields
\begin{align}\label{eq:mollified-transport-integratedInTimeAndSpace}
\iint  \eta([u]^{\eps,\delta})\partial_t\phi + \eta([u]^{\eps,\delta})\partial_x(\phi b) \,dx\,dt = -\iint \phi r_{\eps,\delta}\eta'([u]^{\eps,\delta})\,dx\,dt
\end{align}
We claim that, as $\delta\to0$, the above converges to
\begin{equation}\label{eq:mollified-transport-integrated}
\iint  \eta([u]^{\eps})\partial_t\phi + \eta([u]^{\eps})\partial_x(\phi b) \,dx\,dt = -\iint \phi r_{\eps}\eta'([u]^{\eps})\,dx\,dt
\end{equation}
where
\[
r_{\eps}(x,t) \coloneqq \int_\R \rho_\eps(x-y)\bigl(b(x,t) - b(y,t)\bigr)\,d_yu(y,t).
\]
It is clear that the left-hand side of \eqref{eq:mollified-transport-integratedInTimeAndSpace} converges to that of \eqref{eq:mollified-transport-integrated} by standard arguments. For the right-hand side, we use the ``Fubini theorem'', Lemma~\ref{lem:fubini}, so to move the integral in $t$:
\begin{align}\label{eq: integralWithG}
\iint \phi r_{\eps,\delta}\eta'([u]^{\eps,\delta})\,dx\,dt
= \iiint G_{\eps,\delta}(x,y,s)\,d_yu(y,s)\,ds\,dx
\end{align}
where
\[
G_{\eps,\delta}(x,y,s)\coloneqq \int\phi(x,t) \rho_\eps(x-y)\rho_\delta(t-s)(b(x,t)-b(y,s))
\eta'([u]^{\eps,\delta})(x,t)\,dt.
\]
For a.e.~$x\in\R$ and $s\geq0$, the term $G_{\eps,\delta}(x,y,s)$ converges uniformly in $y$ as $\delta\to0$. Thus, using Theorem~\ref{thm:young-integral-properties}~\ref{thm:young-integral-properties-integration-by-parts} and \ref{thm:young-integral-properties_pw-convergence}, and the dominated convergence theorem, the integral \eqref{eq: integralWithG} converges to the desired limit after we relabel $s\mapsto t$.

Next, we send $\varepsilon\to0$: The left-hand side of \eqref{eq:mollified-transport-integrated} converges to $\iint \eta(u)\partial_t \phi\,dx\,dt + \iint \eta(u)\,d_x(\phi b)\,dt$. Indeed, the first integral converges by Lebesgue's dominated convergence theorem, and the second converges due to Theorem~\ref{thm:young-integral-properties}~\ref{thm:young-integral-properties_pw-convergence}.

As for the right-hand side of \eqref{eq:mollified-transport-integrated}, we claim that it vanishes as $\eps\to0$. Denoting $\eta_\eps' \coloneqq \eta'([u]^\eps)$, we get
\begin{align*}
&\iint \phi r_{\eps}\eta_\eps'\,dx\,dt
= \iiint \rho_\eps(x-y)\bigl(b(x,t)-b(y,t)\bigr)\eta_\eps'(x,t)\phi(x,t)\,d_yu(y,t)\,dx\,dt \\
\intertext{(\emph{integration by parts, Theorem \ref{thm:young-integral-properties} \ref{thm:young-integral-properties-integration-by-parts} and \ref{thm:young-integral-properties_product-rule}})}
&\quad = \iiint (\partial_x\rho_\eps)(x-y)\bigl(b(x,t)-b(y,t)\bigr)\eta_\eps'(x,t)\phi(x,t)u(y,t)\,dy\,dx\,dt \\*
&\qquad + \iiint \rho_\eps(x-y) \eta_\eps'(x,t)\phi(x,t)u(y,t)\,d_yb(y,t)\,dx\,dt \\
\intertext{(\emph{add/subtract $u(x,t)$ in the first integral; apply Fubini in the second})}
&\quad = \iiint (\partial_x\rho_\eps)(x-y)\bigl(b(x,t)-b(y,t)\bigr)\bigl(u(y,t)-u(x,t)\bigr)\eta_\eps'(x,t)\phi(x,t)\,dy\,dx\,dt \\*
&\qquad + \iiint (\partial_x\rho_\eps)(x-y)\bigl(b(x,t)-b(y,t)\bigr)u(x,t)\eta_\eps'(x,t)\phi(x,t)\,dy\,dx\,dt \\*
&\qquad + \iint\Biggl(\int \rho_\eps(x-y) \eta_\eps'(x,t)\phi(x,t)\,dx\Biggr)u(y,t)\,d_yb(y,t)\,dt \\
&\quad\eqqcolon A_\eps+B_\eps+C_\eps.
\end{align*}
We start by estimating $A_\eps$. Let $R>0$ be such that $\supp\phi\subset[-R,R]\times(0,T)$, and let $p,q\geq1$ be such that $1/\theta\coloneqq 1/p+1/q>1$, and $\sup_{t\in[0,T]} |u(t)|_{V^p([-R,R])}<\infty$ and $\sup_{t\in[0,T]}|b(t)|_{V^q([-R,R])}\lesssim 1+\lambda(t)$ for some $\lambda\in L^1([0,T])$ (see Corollary~\ref{cor: finitePVariationOfTheVelocityB}). Then, by \eqref{eq:double-translation-estimate} in Lemma~\ref{lem:l1-translation-error-estimate},
\begin{align*}
|A_\eps| &\leq \|\partial_x\rho_\eps\|_{L^\infty} \|\eta'\|_{L^\infty} \|\phi\|_{L^\infty} \int_{-\eps}^\eps\int_0^T \int_{-R}^R \begin{aligned}[t]&\bigl|b(x-z,t)-b(x,t)\bigr|\\
&\times\bigl|u(x-z,t)-u(x,t)\bigr|\, dx\,dt\,dz\end{aligned} \\
&\lesssim \eps^{-2} \int_{-\eps}^\eps\int_0^T |z| \big|u(t)\big|_{V_{\eps}^\infty}^{1-\theta}\bigl|b(t)\bigr|_{V_{\eps}^\infty}^{1-\theta} \big|u(t)\big|_{V^p}^\theta \big|b(t)\big|_{V^q}^\theta\,dt\,dz \\
&\lesssim \int_0^T \big|b(t)\big|_{V_{\eps}^\infty}^{1-\theta}\bigl(1+\lambda(t)\bigr)^\theta\,dt \\
&\to 0
\end{align*}
as $\eps\to0$, since $x\mapsto b(x,t)$ is continuous for a.e.~$t$.

For $B_\eps$ we get
\begin{align*}
B_\eps &= \iint \Biggl(\underbrace{\int (\partial_x\rho_\eps)(x-y)\,dy}_{=\,0}\Biggr)b(x,t) u(x,t)\eta_\eps'(x,t)\phi(x,t)\,dx\,dt \\
&\quad - \iint \Biggl(\underbrace{\int (\partial_x\rho_\eps)(x-y)b(y,t)\,dy}_{=\,\partial_x [b]^\eps(x)}\Biggr) u(x,t)\eta_\eps'(x,t)\phi(x,t)\,dx\,dt \\
&= - \iint \partial_x [b]^\eps u \eta_\eps'\phi\,dx\,dt,
\end{align*}
and for $C_\eps$ we get
\begin{align*}
C_\eps &= \iint \Biggl(\iint \rho_\eps(x-y) \eta_\eps'(x,t)\phi(x,t)\,dx\Biggr)u(y,t)\,d_yb(y,t)\,dt \\
&= \iint \bigl[\eta_\eps'\phi\bigr]^\eps u\,d_yb(y,t)\,dt.
\end{align*}
Using the ``dominated convergence theorem'' in Theorem \ref{thm:young-integral-properties} \ref{thm:young-integral-properties_pw-convergence}, we get
\begin{align*}
&\lim_{\eps\to0} \iint r_\eps \eta'([u]^\eps)\phi\,dx\,dt
= \lim_{\eps\to0}\big( A_\eps+B_\eps+C_\eps\big) \\
&\qquad = -\iint \eta'(u)u\phi\,d_xb(x,t)\,dt + \iint \eta'(u)\phi u\,d_xb(x,t)\,dt = 0.
\end{align*}

Finally, if $\eta$ is not $C^2$ but merely Lipschitz, we approximate it uniformly by a sequence of $C^2$ functions. Since the integrands in the corresponding weak formulations converge uniformly, Theorem \ref{thm:young-integral-properties} \ref{thm:young-integral-properties_pw-convergence} ensures convergence of each integral, proving that also $\eta\circ u$ is a weak solution.
\end{proof}

\subsubsection{$L^1$ stability of solutions}\label{sec:uniqueness-and-stability}
\begin{theorem}\label{thm:uniqueness-and-stability}
Let $b$ satisfy the conditions of Assumption \ref{ass:b-conditions}, let $T>0$, and let $u_1,u_2$ be two weak solutions of the transport equation \eqref{eq:transport1d}. Then for any $R>0$, and $t\in[0,T]$,
\begin{equation}
\bigl\|u_1(t)-u_2(t)\bigr\|_{L^1([-R,R])} \leq \Psi\Bigl(\bigl\|u_1(0)-u_2(0)\bigr\|_{L^1([-R-Mt,R+Mt]},t\Bigr)
\end{equation}
for some modulus of continuity $\Psi$, and with $M\coloneqq\|b\|_{L^\infty}$.
In particular, the equation has at most one solution for any initial data $u_0\in V^p(\R)$.
\end{theorem}
\begin{proof}
Let $u_1,u_2$ be two solutions of the equation, so that $u\coloneqq u_1-u_2$ satisfies the same equation but with $u(\cdot,0)=u_0\coloneqq u_1(0)-u_2(0)$. By Theorem \ref{thm:renormalizability}, the function $|u|$ satisfies the same equation. The weak formulation for $|u|$ is then
\[
\int_0^\infty \int_\R |u|\partial_t \phi\,dx\,dt + \int_0^\infty\int_\R |u|\,d_x(b\phi)\,dt + \int_\R |u_0|\phi(0)\,dx = 0
\]
for any $\phi\in C_c^\infty(\R\times[0,\infty))$.
By using the $L^1$ time continuity of $u$ from Proposition~\ref{prop:time-continuity}, one can easily show that the above is equivalent to
\begin{equation}\label{eq:weakFormulationTruncated}
\int_0^T \int_\R |u|\partial_t \phi\,dx\,dt + \int_0^T\int_\R |u|\,d_x(b\phi)\,dt = \int_\R |u|\phi(T)\,dx - \int_\R |u_0|\phi(0)\,dx
\end{equation}
for any $\phi\in C_c^\infty(\R\times[0,T])$ and any $T>0$. The second integral can be rewritten using Theorem \ref{thm:young-integral-properties}~\ref{thm:young-integral-properties_product-rule} to get
\begin{align*}
\int_0^T\int_\R |u|\,d_x(b\phi)\,dt
&= \int_0^T\int_\R |u|\phi\,d_xb\,dt + \int_0^T\int_\R |u| b \partial_x\phi\,dx\,dt.
\end{align*}
Pick any $M \geq \|b\|_{L^\infty}$ and $R>0$, and let $0\leq \xi\in C^\infty(\R)$ satisfy $\xi|_{(-\infty,0]}\equiv 1$, $\xi|_{[1,\infty)}\equiv 0$ and $\xi'\leq 0$. Pick some $\eps>0$ and set
\[
\phi(x,t) \coloneqq \xi\biggl(\frac{|x|-R-M(T-t)}{\eps}\biggr) \qquad\text{for } x\in\R,t\in[0,T]
\]
so that $\supp\phi\subset[-\frac{L}{2},\frac{L}{2} ]\times[0,T]$ where $\frac{L}{2}\coloneqq R+MT$. We have now
\begin{align*}
\int_0^T \int_\R |u|\partial_t \phi\,dx\,dt
&= \int_0^T \int_\R |u|\frac{M}{\eps}\xi'\,dx\,dt
\intertext{and}
\int_0^T\int_\R |u| b \partial_x\phi\,dx\,dt
&= \int_0^T\int_\R |u| b \frac{\sign(x)}{\eps}\xi'\,dx\,dt.
\end{align*}
Hence, continuing from \eqref{eq:weakFormulationTruncated}, we get
\begin{align*}
&\int_\R |u|\phi(T)\,dx - \int_\R|u_0|\phi(0)\,dx \\
&\qquad= \int_0^T \int_\R \biggl(\frac{M+b\sign(x)}{\eps}\biggr)|u|\xi'\,dx\,dt + \int_0^T\int_\R |u|\phi\,d_xb\,dt \\
&\qquad\leq \int_0^T\int_\R |u|\phi\,d_xb\,dt
\end{align*}
(since $M\geq\|b\|_{L^\infty}$, and since $\xi'\leq 0$). By assumption~\ref{def: definitionOfWeakSolution-finite-variation} of Definition~\ref{def: definitionOfWeakSolution}, there is some $p\in[1,\infty)$ such that $\sup_{0\leq t\leq T}|u(t)|_{V^p([-\frac{L}{2},\frac{L}{2}])}<\infty$. Hence, for any $\eps>0$ and $t\in[0,T]$, the function $\gamma(x)\coloneqq |u(x,t)|\phi(x,t)$ has bounded $L^1$ norm and $V^p$ seminorm, and satisfies Assumption \ref{ass:b-conditions}. By Theorem~\ref{thm: controlOnIntegralWithPVariationandOnesidedRegularity} with $\gamma=|u(t)|\phi(t)$ and $\beta=b(t)$, we can then estimate further
\[
\int_0^T\int_\R |u|\phi\,d_xb\,dt \leq
\int_0^T \lambda(t)\widetilde\omega_b\big(\|u\phi(t)\|_{L^1}\big)\,dt
\]
where $\lambda\in L^1_\loc$ is the coefficient appearing in \eqref{eq: theOsgoodCondition}, and where $\widetilde\omega_b$ is the Osgood modulus
\[
\widetilde\omega_b(h) = \frac{k^{2p-1}}{h^{2p-2}}\omega_b\bigg(\frac{h^{2p-1}}{k^{2p-1}}\bigg) + c_ph,
\]
with $k\coloneqq \esssup_{t\in[0,T)}\bigl||u|\phi(t)\bigr|_{V^p}L^{1-\frac{1}{p}}$. Note that $k$ is finite because
\[
\bigl||u|\phi(t)\bigr|_{V^p} \leq \|u\|_{L^\infty([-\frac{L}{2},\frac{L}{2}])}\sup_{t\geq0}|\phi(t)|_{TV} + \esssup_{0\leq s<T}|u(s)|_{V^p([-\frac{L}{2},\frac{L}{2}])}\|\phi\|_{L^\infty} < \infty,
\]
for a.e.~$t\in[0,T)$. We obtain
\[
\int_\R |u|\phi(x,T)\,dx \leq \int_\R|u_0|\phi(x,0)\,dx +
\int_0^T \lambda(t)\widetilde\omega_b\big(\|u\phi(t)\|_{L^1}\big)\,dt
\]
for a fixed Osgood modulus $\widetilde\omega_b$. By the Osgood inequality, Lemma \ref{lem:osgood}, we get
\[
\int_\R |u|\phi(x,T)\,dx \leq \Psi\Biggl(\int_\R |u_0|\phi(x,0)\,dx,\,T\Biggr)
\]
for a modulus of continuity $\Psi(\cdot,T)$. In particular, passing $\eps\to0$ yields
\[
\int_{-R}^R |u|(x,T)\,dx \leq \Psi\Biggl(\int_{-R-MT}^{R+MT} |u_0(x,0)|\,dx,\,T\Biggr).
\]
\end{proof}

\section{The backwards problem}\label{sec:backwards1d}
In this section we treat the \emph{backwards problem} for the one-dimensional transport equation,
\begin{equation}\label{eq:backwardsTransportEqn}
\begin{cases}
    \partial_t u + b\partial_x u = 0 & \text{for } x\in\R, t\in(0,T), \\
    u(\cdot,T) = u_T
\end{cases}
\end{equation}
for some Osgood velocity field $b\from\R\times(0,T)\to\R$ and some terminal data $u_T\from\R\to\R$. The challenges for this problem are distinct from those of the forwards problem: Unlike \eqref{eq:transport1d}, there is no risk of discontinuities spontaneously appearing in the solution, but on the downside, there \emph{is} a risk of non-uniqueness of solutions. These issues have been thoroughly investigated for one-sided Lipschitz velocities; see e.g.~\cite{BouJam97}.

\begin{theorem}\label{thm:backwards1d-well-posed}
Let $b$ and $u_T$ satisfy Assumptions \ref{ass:b-conditions} and \ref{ass:uzero-conditions} (with $u_T$ in place of $u_0$), respectively. Let $X=X_t(x,s)$ be the flow of the velocity field $b$. Then the function
\begin{equation}\label{eq:backwardsSolution}
    u(x,t)\coloneqq u_T\bigl(X_T(x,t)\bigr) \qquad \text{for } (x,t)\in\R\times[0,T]
\end{equation}
is a weak solution of the backwards problem \eqref{eq:backwardsTransportEqn}. It is the only weak solution of \eqref{eq:backwardsTransportEqn} that is stable under smooth approximation: If $u^{\eps,\delta}$ solves the equation with mollified velocity $[b]^\eps$ and terminal data $[u_T]^\delta$, then $u^{\eps,\delta}\to u$ pointwise~a.e.\ as $\eps,\delta\to0$.
\end{theorem}
\begin{proof}
The proof that \eqref{eq:backwardsSolution} solves \eqref{eq:backwardsTransportEqn} follows the proof of Theorem \ref{thm:existence-of-solution} closely, but with $X_T(x,t)$ in place of $X_t^{-1}(x)$. (In fact, the proof is simpler, since the forwards flow is continuous, as opposed to the backwards flow.) The fact that this solution is stable under mollification of $u_T$ and $b$ also follows from that proof, as does the claim that it is unique in being stable.
\end{proof}

\section{The forwards, multi-dimensional problem}\label{sec:multi-d}

In this section we treat the multi-dimensional transport equation
\begin{equation}\label{eq:transport-eq-multi-d}
\partial_t u + b\cdot\nabla u = 0, \qquad u(\cdot,0)=u_0
\end{equation}
for some $b\from\R^d\times\R_+\to\R^d$.

\begin{assumption}[The velocity field]\label{ass:b-conditions-multi-d}
We assume that $b\from\R^d\times\R_+\to\R^d$ is measurable and that it satisfies the following properties.
\begin{enumerate}[label=(\Alph*)]
\item Boundedness: $b$ is bounded.
\item Log-Lipschitz continuity: For all distinct points $x,y\in\R^d$ and all $t>0$,
\begin{equation}\label{eq:log-lipschitz-condition}
\frac{\bigl|(b(x,t)-b(y,t))\cdot(x-y)\bigr|}{|x-y|} \leq \ell(|x-y|)
\end{equation}
where
\begin{equation}\label{eq:log-lipschitz-modulus}
\ell(h)\coloneqq\begin{cases}
        h\abs{\log(h)} & \text{for } h\in[0,e^{-1}),\\
        e^{-1} & \text{for } h\geq e^{-1}.
    \end{cases}
\end{equation}
\end{enumerate}
\end{assumption}

\begin{assumption}[The initial data]\label{ass:uzero-conditions-multi-d}
We assume that $u_0\from\R^d\to\R$ is locally H\"older continuous, i.e.~for every compact $K\subset\R^d$ there is some $\alpha\in(0,1)$ such that $u_0\bigr|_K\in C^{0,\alpha}(K)$.
\end{assumption}

\begin{definition}\label{def: definitionOfWeakSolutionInMultipleDimensions}
Let $b$ satisfy Assumption \ref{ass:b-conditions-multi-d} and let $u_0$ satisfy Assumption \ref{ass:uzero-conditions-multi-d}. We then say that $u\in L^\infty(\R^d\times[0,\infty))$ is a \emph{weak solution} of \eqref{eq:transport-eq-multi-d} if
\begin{enumerate}[label=(\roman*)]
    \item it is locally H\"older continuous, i.e.~for every compact $K\subset\R^d\times[0,\infty)$ there is some $\alpha\in(0,1]$ such that $u\bigr|_{K}\in C^{0,\alpha}(K)$
    \item for every $\phi\in C_c^\infty(\R^d\times[0,\infty))$ we have
    \begin{equation}\label{eq:weak-solution-multi-d}
    \begin{split}
        &\int_0^\infty \int_{\R^d} u\partial_t\phi\,dx\,dt + \sum_{i=1}^d\int_0^\infty\int_{\R^{d-1}}\int_\R u\,d_{x_i}(\phi b_i)\,d\hat{x}_i\,dt \\
        &\qquad + \int_{\R^d}\phi(x,0)u_0(x)\,dx = 0
    \end{split}
    \end{equation}
    (as in \eqref{eq: firstPlaceWhereXhatAppears}, we denote $\hat{x}_i = (x_1,\dots,x_{i-1},x_{i+1},\dots,x_d)$).
\end{enumerate}
\end{definition}

The main theorem of this section is the following:
\begin{theorem}\label{thm:uniquenessx-and-stability-multi-d}
    Under Assumptions \ref{ass:b-conditions-multi-d} and \ref{ass:uzero-conditions-multi-d}, the function
    \begin{equation}\label{eq:canonical-solution-multi-d}
    u(x,t)\coloneqq u_0\bigl(X_0(x,t)\bigr), \qquad x\in\R^d,\ t\geq0
    \end{equation}
    is the unique weak solution (in accordance with Definition \ref{def: definitionOfWeakSolutionInMultipleDimensions}) of \eqref{eq:transport-eq-multi-d}. It is stable with respect to $u_0$ in $L^1_\loc$.
\end{theorem}

\begin{remark}\label{rem: onRegularityAssumptionsInMultiD} ~
\begin{enumerate}[label=(\roman*)]
\item
Comparing the regularity condition \eqref{eq:log-lipschitz-condition} with the one-dimensional counterpart \eqref{eq: theOsgoodCondition}, it is evident that \eqref{eq:log-lipschitz-condition} is much stricter: It requires a specific Osgood modulus, and it is a \emph{two-sided condition}. In fact, with some rather easy modifications to the proof, the condition \eqref{eq:log-lipschitz-condition} can be replaced by the non-symmetric two-sided condition
\begin{equation}\label{eq:two-sided-modulus-multi-d}
-\ell(|x-y|) \leq \frac{(b(x,t)-b(y,t))\cdot(x-y)}{|x-y|} \leq \omega(|x-y|) \quad\forall\,x\neq y
\end{equation}
where $\ell$ is given by \eqref{eq:log-lipschitz-modulus}, and where $\omega$ is \emph{any} Osgood modulus satisfying Assumption~\ref{ass:b-conditions}~\ref{ass:oshc}.
The lower bound in \eqref{eq:two-sided-modulus-multi-d} bounds the rate at which the backwards flow can expand; it guarantees the existence, uniqueness and regularity of the backwards flow appearing in \eqref{eq:canonical-solution-multi-d}.
The upper bound in \eqref{eq:two-sided-modulus-multi-d} bounds the rate of contraction of the backwards flow; it guarantees the uniqueness of weak solutions via the same non-linear $L^1$ estimate that we used in Section~\ref{sec:forwards-uniqueness}.
\item
The right-hand side of \eqref{eq:log-lipschitz-condition} can be replaced by $C\ell(|x-y|)$ for any $C>0$ by performing the change of variables $t \mapsto t/C$. For the sake of simplicity we do not consider right-hand sides of the form $\lambda(t)\ell(|x-y|)$ for some $\lambda\in L^1_\loc([0,\infty))$, but this can be handled similarly to the one-dimensional case.

\item\label{rem:multi-d-b-continuous}
Bonicatto and Gusev~\cite[Proposition~5.1]{bonicatto_non-uniqueness_2019} observed that a two-sided Osgood condition like \eqref{eq:log-lipschitz-condition} ensures that $b$ is continuous (although not necessarily log-Lipschitz continuous). Note, however, that \eqref{eq:log-lipschitz-condition} ensures that $b_i$ is log-Lipschitz continuous with respect to $x_i$, and therefore also H\"older continuous in $x_i$ of any exponent $\alpha\in(0,1)$.

\item
The second integral in \eqref{eq:weak-solution-multi-d} is well-defined due to Lemma~\ref{lem:iterated-integral}, and is bounded due to Theorem~\ref{thm:young-integral-properties}~\ref{thm:young-integral-properties_size-control} and the fact that $u$ is locally H\"older continuous and $b_i$ is H\"older continuous with respect to $x_i$.
\end{enumerate}
\end{remark}

A variant of the following is given in \cite[Theorem~3.7]{bahouri_fourier_2011}
\begin{lemma}\label{lem:ode-well-posed-multi-d}
    Under Assumption \ref{ass:b-conditions-multi-d}, there exists a unique flow $X=X_t(x,s)$ for $b$, defined for all $s,t\in\R$ and $x\in\R^d$. The map $x\mapsto X_t(\cdot,s)$ is bijective for all $s,t\in[0,\infty)$ and is locally H\"older continuous; more precisely, $|X_t(x,s)-X_t(y,s)| \leq \Psi(|x-y|,|t-s|)$ for all $x,y,t,s$, where
    \begin{equation}\label{eq:loglip-cont-modulus}
        \Psi(z,\tau) \coloneqq \begin{cases}
            z^{e^{-t}} & \text{if } z\leq e^{-e^t} \\
            \tfrac1e\bigl(1+t-\log(-\log z)\bigr) & \text{if }  e^{-e^t} < z \leq e^{-1} \\
            z+t/e & \text{if } e^{-1}<z.
        \end{cases}
    \end{equation}
\end{lemma}
\begin{proof}
The fact that the forwards-in-time flow exist follows from spatial continuity of $b$ (cf.~Remark~\ref{rem: onRegularityAssumptionsInMultiD}~\ref{rem:multi-d-b-continuous}), and uniqueness follows from the same argument as in Theorem \ref{thm:ode-well-posed}. Since \eqref{eq:log-lipschitz-condition} is a two-sided bound, the same argument can be applied backwards in time, showing that the flow is well-defined both forwards and backwards in time, and is bijective in space. The specific form \eqref{eq:loglip-cont-modulus} of the continuity modulus follows from an elementary but tedious computation of the modulus $\Psi$ in Lemma~\ref{lem:osgood}.
\end{proof}

We are now ready to prove the main theorem.
\begin{proof}[{Proof of Theorem~\ref{thm:uniquenessx-and-stability-multi-d}}]
By Lemma~\ref{lem:ode-well-posed-multi-d}, the flow $X$ is locally H\"older continuous, so $u$ is also locally H\"older continuous. The proof that $u$ solves \eqref{eq:transport-eq-multi-d} is similar to that of Theorem~\ref{thm:existence-of-solution}, but is much simpler because $X$ and $u$ are continuous.

The proof that $u$ is renormalizable, i.e.~that $\eta\circ u$ is a solution for all Lipschitz $\eta\from\R\to\R$, is very similar to that of Theorem \ref{thm:renormalizability}, but is simpler because the variation bounds on $u$ and $b$ are replaced by H\"older and log-Lipschitz continuity. In particular, the analogue of the term $A_\eps$ in the proof of Theorem \ref{thm:renormalizability} is estimated as follows:
\begin{align*}
|A_\eps| &= \Biggl|\int_0^\infty\int_{\R^d}\int_{\R^d}
\begin{aligned}[t]\bigl(b(x,t)-b(y,t)\bigr)\cdot \nabla_x\bigl(\eps^{-d}\rho(\eps^{-1}|x-y|)\bigr)& \\
{}\times \bigl(u(y,t)-u(x,t)\bigr)\eta_\eps'(x,t)\phi(x,t)&\,dy\,dx\,dt\Biggr|
\end{aligned} \\
\intertext{(\textit{as $u$ is locally $\alpha$-H\"older for some $\alpha>0$})}
&\lesssim \eps^{-d-1}\int_0^\infty\int_{\R^d}\int_{\R^d}
\begin{aligned}[t]\biggl|\bigl(b(x,t)-b(y,t)\bigr)\cdot \frac{x-y}{|x-y|}\rho'(\eps^{-1}|x-y|)\biggr|& \\
{}\times |x-y|^\alpha|\phi(x,t)|&\,dy\,dx\,dt
\end{aligned} \\
\intertext{(\textit{using \eqref{eq:log-lipschitz-condition}})}
&\lesssim \eps^{-d-1}\iint_{\supp\phi}\int_{B_\eps(x)}
\ell(|x-y|)|x-y|^\alpha\,dy\,dx\,dt \\
&\lesssim \eps^{-1}\int_0^\eps \abs{\log h} h^\alpha\,dh \to 0
\end{align*}
as $\eps\to0$.

To prove $L^1$ stability and uniqueness, let $u$ be a weak solution of \eqref{eq:transport-eq-multi-d} and fix some $\phi\in C_c^\infty(\R^d\times[0,T])$. Let further $L>0$ be such that $\supp(\phi)\subset[-\frac{L}{2},\frac{L}{2}]^d\times[0,T]$. By assumption, there is some $\alpha\in(0,1]$ such that $u(\cdot, t)$ is $\alpha$-H\"older continuous (uniformly in $t$) on $[-\frac{L}{2},\frac{L}{2}]^d$; in particular, letting $p\coloneqq 1/\alpha$, the function $u(\cdot,t)$ has uniformly bounded $p$-variation in each spatial direction over $[-\frac{L}{2},\frac{L}{2}]^d$. We follow the proof of Theorem \ref{thm:uniqueness-and-stability} to get,
\[
\int_{\R^d} |u|\phi(x,T)\,dx \leq \int_{\R^d}|u_0|\phi(x,0)\,dx + \sum_{i=1}^d
\int_0^T \int_{\R^{d-1}} \widetilde\omega\big(\|u\phi(\hat{x}_i,t)\|_{L^1(\R)}\big)\,d\hat{x}_i\,dt
\]
where
\[
\widetilde\omega(h) \coloneqq \frac{k^{2p-1}}{h^{2p-2}}\ell\biggl(\frac{h^{2p-1}}{k^{2p-1}}\biggr) + c_p h, \qquad k\coloneqq \max_{i=1,\dots,d}\esssup_{\substack{t\in(0,T] \\ \hat{x}_i\in\R^{d-1}}} \bigl||u|\phi(\hat{x}_i,t)\bigr|_{V^p(\R)}L^{1-\frac1p}.
\]
Inserting the expression for $\ell$, it is a straightforwards calculation to show that
\[
\widetilde\omega(h) \leq \widehat\omega(h) \coloneqq (2p-1)k\ell\Bigl(\frac{h}{k}\Bigr) + c_ph.
\]
Note that $\widehat\omega$ (unlike $\widetilde\omega$) is concave.
Hence,
\begin{align*}
&\int_{\R^d} |u|\phi(x,T)\,dx -\int_{\R^d}|u_0|\phi(x,0)\,dx
\leq \int_0^T\sum_{i=1}^d \int_{\R^{d-1}} \widehat\omega\big(\|u\phi(\hat{x}_i,t)\|_{L^1(\R)}\big)\,d\hat{x}_i\,dt \\
&\qquad\leq dL^{d-1}\int_0^T \widehat\omega\Biggl(\frac{1}{dL^{d-1}}\sum_{i=1}^d \int_{\R^{d-1}}\|u\phi(\hat{x}_i,t)\|_{L^1(\R)}\,d\hat{x}_i\Biggr) \,dt\\
&\qquad = dL^{d-1}\int_0^T\widehat\omega\biggl(\frac{1}{L^{d-1}}\|u\phi(\cdot,t)\|_{L^1(\R^d)}\biggr)\,dt.
\end{align*}
Applying Lemma \ref{lem:osgood} as in Theorem \ref{thm:uniqueness-and-stability}, we get $L^1$ stability and uniqueness.
\end{proof}

\section{Vanishing viscosity}\label{sec:vanishing-viscosity}
\newcommand{\E}{\mathbb{E}}
\renewcommand{\P}{\mathbb{P}}

In this section we show that the weak solutions considered in this paper are stable with respect to viscous, parabolic perturbations of the equation. For the sake of simplicity we will only consider the backwards problem. We aim to show that the viscous approximation converges to the ``canonical'' solution $u = u_T\circ X_T$.

Assume that $u_T\in C^0(\R^d)$ and that $b\in C_b^0(\R^d\times[0,T], \R^d)$ satisfies the one-sided Osgood condition
\[
(b(x,t)-b(y,t))\cdot(x-y)\leq|x-y|\omega(|x-y|) \qquad \forall x,y,t
\]
for some Osgood modulus $\omega$.
For some $\eps>0$ consider the viscous backwards problem
\begin{equation}\label{eq:viscous-transport}
\begin{cases}
\partial_t u^\eps + b\cdot \nabla u^\eps + \eps\Delta u^\eps = 0 & \text{for } t\in (0,T),\ x\in\R^d \\
u^\eps(T) = u_T.
\end{cases}
\end{equation}
Consider the associated SDE
\begin{equation}\label{eq:sde}
\begin{cases}
dX_t^\eps = b(X_t^\eps,t)\,dt + \sqrt{2\eps} \,dW_t & \text{for } t>s, \\
X_s^\eps(x,s)=x
\end{cases}
\end{equation}
where $W$ is a Wiener process on some probability space $(\Omega, \mathcal{F},\P)$. Since $b\in L^\infty$, this SDE has a unique strong solution (i.e., a stochastic process adapted to $W$ which satisfies the SDE in the integral sense); see e.g.~\cite{veretennikov_1981,Shap16}. For $x,y\in\R$, the stochastic process $t\mapsto X_t(x,s)-X_t(y,s)$ satisfies the deterministic equation
\[
d\bigl(X_t^\eps(x,s)-X_t^\eps(y,s)\bigr) = \bigl(b(X_t^\eps(x,s),t)-b(X_t^\eps(y,s),t)\bigr)\,dt,
\]
so following the argument of Theorem \ref{thm:ode-well-posed} yields
\[
|X_t^\eps(x,s)-X_t^\eps(y,s)| \leq \Psi\bigl(|x-y|,t-s\bigr) \qquad \text{(for $t\geq s$), almost surely}
\]
for all $x,y\in\R$, for some deterministic, continuous $\Psi\from\R_+\times\R_+\to\R_+$ satisfying $\Psi(z,0)=z$ and $\Psi(0,\tau)=0$ for all $z,\tau$.
Moreover, from \eqref{eq:sde} and the fact that $W$ is $\alpha$-H\"older continuous for any $\alpha<\nicefrac12$, $\P$-almost~surely, it follows that for any $\eta>0$, there is some $\Omega_\eta\subset\Omega$ with $\P(\Omega_\eta)\geq 1-\eta$ such that the $\alpha$-H\"older seminorm of
$t\mapsto X_t^\eps(x,s,\omega)$ is uniformly bounded both with respect to $\omega\in\Omega_\eta$ and $(x,s)\in\R\times[0,T]$.
Finally, estimating differences in $s$ in terms of differences in $t$ and $x$ shows a similar kind of equicontinuity of $s\mapsto X_t^\eps(x,s)$.
Applying Prokhorov's theorem then yields a subsequence $(X^{\eps_n})_{n\in\N}$ converging to some process $X=X_t(x,s,\omega)$, uniformly on compacts in $(x,t,s)$, and in distribution in $\omega$.
Taking limits in \eqref{eq:sde} and using an argument as in e.g.~\cite{BOQ09}, we find that the limit $X$ is the unique solution to
\[
\dot{X}_t(x,s)=b(X_t(x,s),t) \quad \text{(for $t>s$)}, \qquad X_s(x,s)=x.
\]
In particular, the entire sequence $(X^\eps)_{\eps>0}$ converges to $X$.

By a standard argument using the Feynman--Kac formula,
the function $u^\eps(x,t)\coloneqq \E[u_T(X_T^\eps(x,t))]$ is the unique bounded, classical solution of the backwards equation \eqref{eq:viscous-transport}.  Since $u_T$ is assumed to be continuous, the convergence $X^\eps\to X$ ensures that $u^\eps \to u_T\circ X_t$ pointwise.



\section{Inhomogeneous equations}\label{sec:inhomogeneous}
With the theory developed so far, it is straightforwards to treat inhomogeneous problems. We claim that the usual Duhamel formula provides the unique solution to the problem.

For the sake of concreteness we only consider the one-dimensional forwards problem, but the backwards problem and the multi-dimensional problem (from Section~\ref{sec:backwards1d} and \ref{sec:multi-d}, respectively) can be handled similarly. Thus, we consider the inhomogeneous equation
\begin{equation}\label{eq:forwards-inhomo-1d}
\begin{cases}
\partial_t u + b\partial_x u = h \\
u(0) \equiv 0
\end{cases}
\end{equation}
for some $h=h(x,t)$, and where $b$ satisfies Assumption~\ref{ass:b-conditions}.
(We have already treated the Cauchy problem, so we can assume without loss of generality that $u(0)\equiv0$.)

Let $h\from\R\times\R_+\to\R$ be bounded, measurable, and satisfy $\sup_{t\in[0,T]}|h(t)|_{V^p(\R)} <\infty$ for every $T>0$ (so Assumption \ref{ass:uzero-conditions} is enforced for $x\mapsto h(x,t)$, uniformly in $t$).
Let $X=X_t(x,s)$ for $0\leq s\leq t$ and $x\in\R$ be the forwards flow of $b$, and denote the backwards flow by $X_s(x,t) \coloneqq (X_t(\cdot,s))^{-1}(x)$, which is well-defined for a.e.~$x\in\R$, for every $s\leq t$. Then the function $u_s(x,t)\coloneqq h(X_s(x,t),s)$ is well-defined at almost every $x$, for all $s\leq t$, and by Theorem \ref{thm:forwards-problem-well-posed-1d}, $u_s$ solves
\[
\partial_t u_s + b\partial_x u_s = 0 \quad (x\in\R,\ t>s), \qquad u_s(x,s)=h(x,s) \quad (x\in\R)
\]
(in the weak sense). Define $u$ by the Duhamel formula
\begin{equation}\label{eq:duhamel}
    u(x,t) \coloneqq \int_0^t u_s(x,t)\,ds = \int_0^t h(X_s(x,t),s)\,ds.
\end{equation}
Then $u$ is well-defined at almost every $x\in\R$ for all $t\geq0$. By an approximation argument very similar to that in the proof of Theorem~\ref{thm:existence-of-solution}, we find that \eqref{eq:duhamel} is indeed a weak solution of \eqref{eq:forwards-inhomo-1d}.


\section*{Acknowledgements}
We gratefully acknowledge the support of the Research Council of Norway through the project \textit{INICE} (301538). The second author is supported by the French Agence Nationale de la Recherche (ANR) under grant number ANR-23-CE40-0015 (ISAAC).
The first author would like to thank Magnus~Ch.~Ørke for helpful suggestions and feedback for Section~\ref{sec:vanishing-viscosity}.

\printbibliography{}
\end{document}